\documentclass[hidelinks,onefignum,onetabnum]{siamart220329}

\usepackage{lipsum}
\usepackage{amsfonts}
\usepackage{amssymb}
\usepackage{graphicx}
\usepackage{epstopdf}
\usepackage{algorithmic}
\ifpdf
  \DeclareGraphicsExtensions{.eps,.pdf,.png,.jpg}
\else
  \DeclareGraphicsExtensions{.eps}
\fi

\newsiamremark{remark}{Remark}
\newsiamremark{hypothesis}{Hypothesis}
\crefname{hypothesis}{Hypothesis}{Hypotheses}
\newsiamthm{claim}{Claim}

\headers{Monotone inclusion methods for mean-field games}{L. Nurbekyan, S. Liu, and Y. T. Chow}

\title{Monotone inclusion methods for a class of second-order non-potential mean-field games\thanks{Submitted to the editors DATE.
\funding{S. Liu was supported by Air Force Office of Scientific Research (AFOSR) MURI Grant FA9550-18-502 and Office of Naval Research (ONR) Grant N00014-20-1-2787; Y. T. Chow was supported by the Regents Faculty Fellowship and Omnibus Research and Travel Grant from the University of California Riverside.}}}

\author{Levon Nurbekyan\thanks{Department of Mathematics, Emory University, Atlanta, GA 
  (\email{lnurbek@emory.edu}, \url{https://sites.google.com/view/lnurbek/home}).}
\and Siting Liu\thanks{Department of Mathematics, University of California Los Angeles, Los Angeles, CA 
  (\email{siting6@math.ucla.edu}, \url{https://sites.google.com/view/siting6ucla/home}).}
\and Yat Tin Chow\thanks{Department of Mathematics, University of California Riverside, Riverside, CA 
  (\email{yattinc@ucr.edu}, \url{https://ytchow.github.io}).}}

\usepackage{amsopn}

\ifpdf
\hypersetup{
  pdftitle={An Example Article},
  pdfauthor={L. Nurbekyan, S. Liu, and Y. T. Chow}
}
\fi

\begin{document}

\maketitle

\begin{abstract}
We propose a monotone splitting algorithm for solving a class of second-order non-potential mean-field games. Following~\cite{achdou10}, we introduce a finite-difference scheme and observe that the scheme represents first-order optimality conditions for a primal-dual pair of monotone inclusions. Based on this observation, we prove that the finite-difference system obtains a solution that can be provably recovered by an extension of the celebrated primal-dual hybrid gradient (PDHG) algorithm.
\end{abstract}



\section{Introduction}

The main goal of the paper is to introduce a new algorithm for computing the solutions of the following system of PDE
\begin{equation}\label{eq:mfg}
\begin{cases}
    -\partial_t \phi-\nu \Delta \phi +H(t,x,\nabla \phi,\rho)=f(t,x,\rho),\\
    \partial_t \rho-\nu \Delta \rho -\nabla \cdot \left(\rho \nabla_q H(t,x,\nabla \phi,\rho) \right)=0,\\
    \rho(0,x)=\rho_0(x),\quad \phi(T,x)=g(x,\rho(T,\cdot)).
\end{cases}
\end{equation}
Here, we assume periodic boundary conditions; that is, $(t,x) \in [0,T] \times \mathbb{T}^d$, where $\mathbb{T}^d=\mathbb{R}^d/\mathbb{Z}^d$ is the $d$-dimensional flat torus. Furthermore, $\nu \geq 0$ is the noise (viscosity) parameter, $H$ is the Hamiltonian, $f$ is the mean-field coupling function, and $g$ is the terminal cost function.

System~\cref{eq:mfg} characterizes an equilibrium configuration for a continuum of agents that play a non-cooperative differential game. Such games are called mean-field games (MFG) and were independently introduced in~\cite{LasryLions06a,LasryLions06b,LasryLions2007} and \cite{HCM06,HCM07}. In this context, $\rho(t,x)$ and $\phi(t,x)$ are, respectively, the distribution and optimal cost of the agents at time $t$ and location $x$.

We introduce a numerical method for~\eqref{eq:mfg} in the spirit of Benamou-Brenier technique for solving optimal transportation and MFG systems~\cite{BenamouBrenier2000,bencar'15,bencarsan'17,silva18,silva19}. In these works, the authors observe that when $H$ is separable; that is,  $H(x,\rho,q)=H_0(x,q)-f(x,\rho)$, and suitable convexity, monotonicity, and structural assumptions are met,~\eqref{eq:mfg} can be seen as a first-order optimality condition for a convex-concave saddle-point problem. Such MFG are called \textit{variational} or \textit{potential}. Hence, one can use various convex optimization algorithms for computing the solutions of~\eqref{eq:mfg} such as the alternating direction method of multipliers (ADMM)~\cite{BenamouBrenier2000,bencar'15,bencarsan'17} and primal-dual hybrid gradient (PDHG)~\cite{silva18,silva19}.

Here, we go beyond the potential and separable settings and provide a version of PDHG algorithm to solve~\eqref{eq:mfg}. Our essential observation is that under the so-called Lasry-Lions monotonicity condition~\eqref{eq:mfg} can be seen as a primal-dual pair of monotone inclusions where the monotone maps are not subdifferential maps in general. We then solve~\eqref{eq:mfg} by a PDHG variant for monotone inclusions~\cite{vu13}. The possibility of solving non-potential MFG using monotone inclusion variants of PDHG was hinted on in~\cite{liu2020splitting} for nonlocal MFG systems.

For related work on non-separable MFG systems we refer to~\cite{Lauriere2023,camilli2023convergence} for policy iteration and Newton's methods, and~\cite{Almulla2017two,gomes20hessian,Gomes2021numerical} for monotone flows. In~\cite{Lauriere2023,camilli2023convergence}, the analyses strongly rely on the ellipticity and do not handle first-order systems ($\nu=0$) and non-smooth mean-field interactions ($\epsilon=0$ in~\cref{eq:H}). In contrast, we expect our methods to extend to these singular cases mutatis mutandis due to the variational nature of our techniques. Numerical experiments in~\cref{subsec:viscosity=0} and~\cref{subsec:cong=0} support our claim.

Works~\cite{Almulla2017two,gomes20hessian,Gomes2021numerical} address finite-state and stationary problems. Additionally, the algorithms in these works are forward only and do not come with convergence guarantees for the time-discretizations of the monotone flows.

\section{Numerical analysis}

\subsection{A finite-difference scheme}

We follow~\cite{achdou10,silva19} for introducing a semi-implicit scheme for~\eqref{eq:mfg}. For simplicity, we assume that $d=2$. Let $h,\Delta t >0$ be such that $N_h=\frac{1}{h} \in \mathbb{N}$ and $N_T=\frac{T}{\Delta t} \in \mathbb{N}$. We then introduce uniform space-time grids
\begin{equation*}
    x_{ij}=(ih, jh),\quad t_k=k \Delta t,\quad 0\leq i,j \leq N_h-1,~0\leq k \leq N_T.
\end{equation*}
To enforce periodicity, we assume that $x_{ij}=x_{i'j'}$ whenever $i\equiv i'~(\operatorname{mod} N_h)$ and $j\equiv j'~(\operatorname{mod} N_h)$. Next, for a grid-function $f_{ij}^k=f(x_{ij},t_k)$ we denote by
\begin{equation*}
\begin{split}
    (D_1 f^k)_{ij}=&\frac{f^k_{(i+1)j}-f^k_{ij}}{h},\quad (D_2 f^k)_{ij}=\frac{f^k_{i(j+1)}-f^k_{ij}}{h},\\
    &\\
    [D_h f^k]_{ij}=&\left((D_1 f^k)_{ij},(D_1 f^k)_{(i-1)j},(D_2 f^k)_{ij},(D_2 f^k)_{i(j-1)} \right),\\
    &\\
    (\Delta_h f^k)_{ij}=&\frac{f^{k}_{(i-1)j}+f^{k}_{(i+1)j}+f^{k}_{i(j-1)}+f^{k}_{i(j+1)}-4f^{k}_{ij}}{4h^2 },\\
    &\\
    (D_t f_{ij})^k=&\frac{f_{ij}^{k+1}-f_{ij}^k}{\Delta t}.
\end{split}
\end{equation*}
We then introduce a discretization of the Hamiltonian $H_h(t,x,q_1,q_2,q_3,q_4,\rho)$ that satisfies the following conditions
\begin{itemize}
    \item \textit{Monotonicity:} $H_h$ is nonincreasing in $q_1,q_3$ and nondecreasing in $q_2,q_4$.
    \item \textit{Consistency:} $H_h(t,x,q_1,q_1,q_2,q_2,\rho)=H(t,x,q,\rho)$ for all $t\in [0,T]$, $x\in \mathbb{T}^2$, $\rho\geq 0$, and $q=(q_1,q_2) \in \mathbb{R}^2$.
    \item \textit{Differentiability:} $\nabla_q H_h$ is continuous.
    \item \textit{Convexity:} $(q_1,q_2,q_3,q_4) \mapsto H_h(t,x,q_1,q_2,q_3,q_4,\rho)$ is convex.
    \item \textit{Lasry-Lions monotonicity:} A structural condition that yields existence and uniqueness of solutions of~\eqref{eq:mfg} is the so called Lasry-Lions monotonicity; that is,
    \begin{equation}\label{eq:LL}
        \begin{pmatrix}
            -\partial_\rho H(t,x,q,\rho) & \frac{\rho \nabla_q^\top H(t,x,q,\rho)}{2}\\
            \frac{\rho \nabla_q H(t,x,q,\rho)}{2} &\rho \nabla^2_q H(t,x,q,\rho)
        \end{pmatrix} \geq 0,
    \end{equation}
    for all $t\in [0,T]$, $(x,q)\in \mathbb{T}^2\times \mathbb{R}^2$, $\rho\geq 0$. We require the same condition on the discretized Hamiltonian; that is,
    \begin{equation}\label{eq:LL_discrete}
        \begin{pmatrix}
            -\partial_\rho H_h(t,x,q_1,q_2,q_3,q_4,\rho) & \frac{\rho \nabla_q^\top H(t,x,q_1,q_2,q_3,q_4,\rho)}{2}\\
            \frac{\rho \nabla_q H(t,x,q_1,q_2,q_3,q_4,\rho)}{2} &\rho \nabla^2_q H(t,x,q_1,q_2,q_3,q_4,\rho)
        \end{pmatrix} \geq 0,
    \end{equation}
    for all $t\in [0,T]$, $(x,q_1,q_2,q_3,q_4)\in \mathbb{T}^2\times \mathbb{R}^4$, $\rho\geq 0$.
\end{itemize}
With these ingredients at hand, the discretization of~\eqref{eq:mfg} introduced in~\cite{achdou10} reads as follows:
\begin{equation}\label{eq:mfg_discrete}
\begin{cases}
    -(D_t \phi_{ij})^k-\nu (\Delta_h \phi^k)_{ij}+H_h(t_k,x_{ij},[D_h \phi^k]_{ij},\rho_{ij}^{k+1})=f(t_k,x_{ij},\rho^{k+1}_{ij})\\
    (D_t \rho_{ij})^k-\nu (\Delta_h \rho^{k+1})_{ij}-\mathcal{B}_{ij}(t_k,\phi^k,\rho^{k+1},\rho^{k+1})=0,\\
    \rho^0_{ij}=(\rho_0)_{ij},~ \phi^{N_T}_{ij}=g(x_{ij},\rho^{N_T}_{ij}),\quad 0\leq i,j \leq N_h-1,~ 0\leq k \leq N_T-1,
\end{cases}
\end{equation}
where
\begin{equation*}
    (\rho_0)_{ij}=\frac{1}{h^2} \int\limits_{\|x-x_{ij}\|_\infty \leq h/2} \rho_0(x)dx.
\end{equation*}
and
\begin{equation*}
\begin{split}
    &\mathcal{B}_{ij}(t,\phi,\rho,\eta)\\
    =&\frac{1}{h}\left(\rho_{ij}\partial_{q_1}H_h(t,x_{ij},[D_h \phi]_{ij},\eta_{ij})-\rho_{(i-1)j}\partial_{q_1}H_h(t,x_{(i-1)j},[D_h \phi]_{(i-1)j},\eta_{(i-1)j}) \right)\\
    +&\frac{1}{h}\left(\rho_{(i+1)j}\partial_{q_2}H_h(t,x_{(i+1)j},[D_h \phi]_{(i+1)j},\eta_{(i+1)j})-\rho_{ij}\partial_{q_2}H_h(t,x_{ij},[D_h \phi]_{ij},\eta_{ij}) \right)\\
    +&\frac{1}{h}\left(\rho_{ij}\partial_{q_3}H_h(t,x_{ij},[D_h \phi]_{ij},\eta_{ij})-\rho_{i(j-1)}\partial_{q_3}H_h(t,x_{i(j-1)},[D_h \phi]_{i(j-1)},\eta_{i(j-1)}) \right)\\
    +&\frac{1}{h}\left(\rho_{i(j+1)}\partial_{q_4}H_h(t,x_{i(j+1)},[D_h \phi]_{i(j+1)},\eta_{i(j+1)})-\rho_{ij}\partial_{q_4}H_h(t,x_{ij},[D_h \phi]_{ij},\eta_{ij}) \right).
\end{split}
\end{equation*}

\subsection{A discrete energy}

Our goal is to formulate the discrete system~\eqref{eq:mfg_discrete} as a primal-dual pair of monotone inclusions. Following~\cite{silva19}, we denote by $\mathcal{M}=\mathbb{R}^{(N_T+1)\times N_h\times N_h}$, $\mathcal{W}=(\mathbb{R}^4)^{N_T\times N_h\times N_h}$, $\mathcal{U}=\mathbb{R}^{N_T \times N_h \times N_h}$ and introduce operators $A:\mathcal{M} \to \mathcal{U}$, $B:\mathcal{W} \to \mathcal{U}$ as follows:
\begin{equation*}
    \begin{split}
        (A\rho)^k_{ij}=&(D_t \rho_{ij})^k-\nu (\Delta \rho^{k+1})_{ij},\\
        (B w)^k_{ij}=&(D_1 w^{k,1})_{(i-1)j}+(D_1 w^{k,2})_{ij}+(D_2 w^{k,3})_{i(j-1)}+(D_2 w^{k,4})_{ij},
    \end{split}
\end{equation*}
for $0\leq i,j \leq N_h-1$ and $0\leq k \leq N_T-1$. Direct calculations~\cite{achdou10,silva19} show that
\begin{equation*}
\begin{split}
    (B^*\phi)^k_{ij}=&-[D_h \phi^k]_{ij},\quad 0\leq k \leq N_T-1,\\
    (A^*\phi)^0_{ij}=&-\frac{1}{\Delta t} \phi^0_{ij},\\
    (A^*\phi)^k_{ij}=&-(D_t \phi_{ij})^{k-1}-\nu (\Delta_h \phi^{k-1})_{ij},\quad 1\leq k \leq N_T-1,\\
    (A^*\phi)^{N_T}_{ij}=&\frac{\phi^{N_T-1}_{ij}}{\Delta t}-\nu (\Delta_h \phi^{N_T-1})_{ij}.
\end{split}
\end{equation*}
Furthermore, we denote by $L_h$ the Legendre dual of $H_h$ defined as
\begin{equation}\label{eq:H_h^*}
    L_h(t,x,v_1,v_2,v_3,v_4,\eta)=\sup_{q_1,q_2,q_3,q_4}\left( - \sum_{i=1}^4 q_i v_i-H_h(t,x,q_1,q_2,q_3,q_4,\eta) \right),
\end{equation}
and denote by $E_h$ the \textit{perspective function}~\cite{combettes18perpsective} of $L_h$; that is, 
\begin{equation}\label{eq:E_h}
    E_h(t,x,\rho,w,\eta)=\begin{cases}
        \rho L_h \left(t, x,\frac{w}{\rho},\eta\right),\quad \rho>0,\\
        (\operatorname{rec}L_h(t,x,\cdot,\eta))(w),\quad \rho=0,\\
        \infty,\quad \text{otherwise},
    \end{cases}
\end{equation}
where the \textit{recession function} $\operatorname{rec}L_h(t,x,\cdot,\eta)$ is defined as
\begin{equation*}
    (\operatorname{rec}L_h(t,x,\cdot,\eta))(w)=\lim\limits_{s \to \infty}\frac{L_h(t,x,v_0+s w,\eta)}{s}.
\end{equation*}
Above, $v_0 \in \mathbb{R}^4$ is an arbitrary point such that $L_h(t,x,v_0,\eta)<\infty$.

Next, we define $J:\mathcal{M} \times \mathcal{W} \times \mathcal{M} \to \mathbb{R}\cup\{\infty\}$ as follows:
\begin{equation}\label{eq:J}
\begin{split}
    J(\rho,w,\eta)=&\sum\limits_{\substack{0\leq i,j \leq N_h-1\\ 0\leq k\leq N_T-1}} E_h(t_k,x_{ij},\rho^{k+1}_{ij},w^k_{ij},\eta^{k+1}_{ij})+\sum\limits_{\substack{0\leq i,j \leq N_h-1\\ 0\leq k\leq N_T-1}} F(t_k,x_{ij},\rho^{k+1}_{ij})\\
    &+\frac{1}{\Delta t}\sum\limits_{0\leq i,j \leq N_h-1} G(x_{ij}, \rho_{ij}^{N_T})+\mathbf{1}_{\rho^0=\rho_0},
\end{split}
\end{equation}
where $\partial_\rho F(t,x,\rho)=f(t,x,\rho)$, and $\partial_\rho G(x,\rho)=g(x,\rho)$. Additionally, $\mathbf{1}_E$ is the convex characteristic function defined as 
\begin{equation*}
\mathbf{1}_E(y)=\begin{cases}
    0,\quad y \in E,\\
    \infty,\quad y \notin E.
\end{cases}    
\end{equation*}

\begin{remark}\label{rmk:perspective_function}
A remarkable property of perspective functions is that $E_h(t,x,\cdot,\cdot,\eta)$ is a convex lower semicontinuous function~\cite{combettes18perpsective}.    
\end{remark}

\subsection{A congestion model}

For concreteness, we consider the MFG model with congestion discussed in~\cite{achdou18}. More specifically, assume that
\begin{equation}\label{eq:H}
    H(t,x,q,\eta)=\frac{|q|^\beta}{\beta(\eta+\epsilon)^\alpha},
\end{equation}
for some $\epsilon>0$, and 
\begin{equation}\label{eq:ab}
    1<\beta \leq 2,\quad 0<\alpha\leq \frac{4(\beta-1)}{\beta}.
\end{equation}
These conditions together with the monotonicity of $\rho \mapsto f(t,x,\rho)$ and $\rho \mapsto g(x,\rho)$, and suitable technical assumptions yield the existence and uniqueness of weak solutions for~\eqref{eq:mfg} as analyzed in~\cite{achdou18}.

Following~\cite{achdou10,silva19}, we discretize $H$ in~\eqref{eq:H} as follows
\begin{equation}\label{eq:H_h}
    H_h(t,x,q_1,q_2,q_3,q_4,\eta)=\frac{((q_1^-)^2+(q_2^+)^2+(q_3^-)^2+(q_4^+)^2)^{\beta/2}}{\beta(\eta+\epsilon)^\alpha}
\end{equation}
and denote by $K=\mathbb{R}_+\times \mathbb{R}_- \times \mathbb{R}_+ \times \mathbb{R}_-$.

\begin{lemma}\label{lma:L_h}
    Assume that $H_h$ is given by~\eqref{eq:H_h}. Then for every $(t,x)\in \mathbb{R}\times \mathbb{T}^d$ and $\eta\geq 0$ we have that
    \begin{equation}\label{eq:L_h}
        L_h(t,x,v_1,v_2,v_3,v_4,\eta)=\begin{cases}
            (\eta+\epsilon)^{\alpha(\beta'-1)} \frac{(v_1^2+v_2^2+v_3^2+v_4^2)^{\beta'/2}}{\beta'},\quad (v_1,v_2,v_3,v_4) \in K,\\
            \infty,\quad \text{otherwise},
        \end{cases}
    \end{equation}
    where $\beta'=\frac{\beta}{\beta-1}$. Furthermore, we have that
    \begin{equation}\label{eq:E_h_congestion}
        E_h(t,x,\rho,w,\eta)=\begin{cases}
            (\eta+\epsilon)^{\alpha(\beta'-1)} \frac{|w|^{\beta'}}{\beta' \rho^{\beta'-1}},\quad \rho>0,~w\in K,\\
            0,\quad (\rho,w)=(0,0),\\
            \infty,\quad \text{otherwise},
        \end{cases}
    \end{equation}
    where $w=(w_1,w_2,w_3,w_4)$.
\end{lemma}
\begin{proof}
    We start by computing $L_h$. Suppose that $(v_1,v_2,v_3,v_4) \notin K$; for instance, let $v_1<0$. Then we have that
    \begin{equation*}
        \begin{split}
            L_h(t,x,v_1,v_2,v_3,v_4,\eta) \geq \sup_{q_1>0} \left(-q_1 v_1-H_h(t,x,q_1,0,0,0,\eta) \right)= \sup_{q_1>0} \left(-q_1 v_1\right)=\infty.
        \end{split}
    \end{equation*}
    Now assume that $(v_1,v_2,v_3,v_4) \in K$. The first order optimality conditions in the concave program~\eqref{eq:H_h^*} yield
    \begin{equation*}
        v_i=-\partial_{q_i} H_h(t,x,q_1,q_2,q_3,q_4,\eta),\quad 1\leq i \leq 4.
    \end{equation*}
    Furthermore, we have that
    \begin{equation*}
    \begin{split}
        \partial_{q_i} H_h(t,x,q_1,q_2,q_3,q_4,\eta)=&-\frac{((q_1^-)^2+(q_2^+)^2+(q_3^-)^2+(q_4^+)^2)^{\beta/2-1}}{(\eta+\epsilon)^\alpha}q_i^{-},\quad i=1,3,\\
        \partial_{q_i} H_h(t,x,q_1,q_2,q_3,q_4,\eta)=&\frac{((q_1^-)^2+(q_2^+)^2+(q_3^-)^2+(q_4^+)^2)^{\beta/2-1}}{(\eta+\epsilon)^\alpha}q_i^{+},\quad i=2,4.
    \end{split}
    \end{equation*}
    Hence, the first order optimality conditions yield
    \begin{equation*}
    \begin{split}
        v_i=&\frac{((q_1^-)^2+(q_2^+)^2+(q_3^-)^2+(q_4^+)^2)^{\beta/2-1}}{(\eta+\epsilon)^\alpha}q_i^{-}\geq 0,\quad i=1,3,\\
        v_i=&-\frac{((q_1^-)^2+(q_2^+)^2+(q_3^-)^2+(q_4^+)^2)^{\beta/2-1}}{(\eta+\epsilon)^\alpha}q_i^{+}\leq 0,\quad i=2,4,
    \end{split}
    \end{equation*}
    and
    \begin{equation*}
    \begin{split}
        v_1^2+v_2^2+v_3^2+v_4^2=&\frac{((q_1^-)^2+(q_2^+)^2+(q_3^-)^2+(q_4^+)^2)^{\beta-1}}{(\eta+\epsilon)^{2\alpha}},\\
        -q_1v_1-q_2v_2-q_3v_3-q_4v_4=&\frac{((q_1^-)^2+(q_2^+)^2+(q_3^-)^2+(q_4^+)^2)^{\beta/2}}{(\eta+\epsilon)^{\alpha}}.
    \end{split}
    \end{equation*}
    Therefore, we obtain
    \begin{equation*}
        \begin{split}
            L_h(t,x,v_1,v_2,v_3,v_4)=&\left(1-\frac{1}{\beta}\right) \frac{((q_1^-)^2+(q_2^+)^2+(q_3^-)^2+(q_4^+)^2)^{\beta/2}}{(\eta+\epsilon)^{\alpha}}\\
            =&\left(1-\frac{1}{\beta}\right)\frac{(\eta+\epsilon)^{\frac{\alpha \beta}{\beta-1}}(v_1^2+v_2^2+v_3^2+v_4^2)^{\frac{\beta}{2(\beta-1)}}}{(\eta+\epsilon)^{\alpha}}\\
            =&(\eta+\epsilon)^{\alpha(\beta'-1)} \frac{(v_1^2+v_2^2+v_3^2+v_4^2)^{\beta'/2}}{\beta'}.
        \end{split}
    \end{equation*}
    Next, we need to compute the recession function of $L_h(t,x,\cdot,\eta)$. Since
    \[
    L_h(t,x,0,0,0,0,\eta)=0,
    \]
    we have that $v_0=(0,0,0,0)\in \operatorname{dom}L_h(t,x,\cdot,\eta)$, and so we can compute the recession function via
    \begin{equation*}
        (\operatorname{rec} L_h(t,x,\cdot,\eta))(w)=\lim\limits_{s\to \infty} \frac{L_h(t,x,sw,\eta)}{s}.
    \end{equation*}
    Since $L_h(t,x,w,\eta)=\infty$ for $w\notin K$ we have that
    \begin{equation*}
        (\operatorname{rec} L_h(t,x,\cdot,\eta))(w)=\infty,\quad \forall w \notin K.
    \end{equation*}
    Furthermore, for $w\in K$ we have that
    \begin{equation*}
        \begin{split}
            (\operatorname{rec} L_h(t,x,\cdot,\eta))(w)=&\lim\limits_{s\to \infty} \frac{L_h(t,x,sw,\eta)}{s}=\lim\limits_{s\to \infty} (\eta+\epsilon)^{\alpha(\beta'-1)}\frac{s^{\beta'}|w|^{\beta'}}{s}\\
            =&\begin{cases}
                0,\quad w=0,\\
                \infty,\quad w\neq 0.
            \end{cases}
        \end{split}
    \end{equation*}
    Summarizing, we find that
    \begin{equation*}
        (\operatorname{rec} L_h(t,x,\cdot,\eta))(w)=\begin{cases}
                0,\quad w=0,\\
                \infty,\quad w\neq 0,
            \end{cases}
    \end{equation*}
    and~\eqref{eq:E_h_congestion} follows readily.
\end{proof}

In~\cite{achdou18}, the authors point out that conditions~\eqref{eq:ab} yield that $H$ satisfies the Lasry-Lions monotonicity condition~\eqref{eq:LL}. Here, we show that $H_h$ preserves that property.
\begin{lemma}\label{lma:LL_discrete}
    Let $H_h$ be as in~\eqref{eq:H_h}, and~\eqref{eq:ab} hold. Then $H_h$ satisfies the Lasry-Lions monotonicity condition~\eqref{eq:LL_discrete}.
\end{lemma}
\begin{proof}
    Denoting by
    \begin{equation}\label{eq:psi}
        \psi(q_1,q_2,q_3,q_4)=\frac{((q_1^-)^2+(q_2^+)^2+(q_3^-)^2+(q_4^+)^2)^{\beta/2}}{\beta},
    \end{equation}
    we have that
    \begin{equation*}
        H_h(t,x,q,\rho)=(\rho+\epsilon)^{-\alpha}\psi(q),
    \end{equation*}
    where we denote by $q=(q_1,q_2,q_3,q_4)$. Then we have that
    \begin{equation*}
        \begin{split}
            \partial_\rho H_h= -\alpha (\rho+\epsilon)^{-\alpha-1} \psi(q),\quad \nabla_q H_h=(\rho+\epsilon)^{-\alpha} \nabla \psi(q),\\
            \rho \nabla_q \partial_\rho H_h=-\alpha \rho (\rho+\epsilon)^{-\alpha-1} \nabla \psi(q),\quad \rho \nabla^2_q H_h=\rho(\rho+\epsilon)^{\alpha} \nabla^2 \psi(q).
        \end{split}
    \end{equation*}
    Hence, we have to prove that
    \begin{equation*}
        \begin{pmatrix}
            \alpha (\rho+\epsilon)^{-\alpha-1} \psi(q) & -\frac{1}{2}\alpha \rho (\rho+\epsilon)^{-\alpha-1} \nabla^\top \psi(q)\\
            -\frac{1}{2}\alpha \rho (\rho+\epsilon)^{-\alpha-1} \nabla \psi(q) & \rho(\rho+\epsilon)^{\alpha} \nabla^2 \psi(q)
        \end{pmatrix} \geq 0
    \end{equation*}
    for all $\rho \geq 0$ and $q\in \mathbb{R}^4$. Note that the inequality is trivial when $\rho=0$ because $\psi(q)\geq 0$ for all $q\in \mathbb{R}^4$. Hence, we can assume that $\rho>0$. Factoring out a positive number $\alpha \rho (\rho+\epsilon)^{-\alpha-1}$ we arrive at an equivalent inequality
    \begin{equation*}
        \begin{pmatrix}
            \frac{\psi(q)}{\rho} & -\frac{\nabla^\top \psi(q)}{2}\\
            -\frac{\nabla \psi(q)}{2} & \frac{\rho+\epsilon}{\alpha} \nabla^2 \psi(q)
        \end{pmatrix} \geq 0.
    \end{equation*}
    Since $\psi$ is convex, we have that $\nabla^2 \psi(q) \geq 0$. Additionally, from~\eqref{eq:ab} we have that
    \begin{equation*}
        \frac{\rho+\epsilon}{\alpha}\geq \frac{\beta \rho}{4(\beta-1)},
    \end{equation*}
    and so it is sufficient to prove that
    \begin{equation*}
        \begin{pmatrix}
            \frac{\psi(q)}{\rho} & -\frac{\nabla^\top \psi(q)}{2}\\
            -\frac{\nabla \psi(q)}{2} & \frac{\beta \rho}{4(\beta-1)} \nabla^2 \psi(q)
        \end{pmatrix} \geq 0.
    \end{equation*}
    Next, since the previous inequality must hold for all $q\in \mathbb{R}^4$, we obtain an equivalent statement if we replace $q$ by $\rho q$; that is,
    \begin{equation*}
        \begin{pmatrix}
            \frac{\psi(\rho q)}{\rho} & -\frac{\nabla^\top \psi(\rho q)}{2}\\
            -\frac{\nabla \psi(\rho q)}{2} & \frac{\beta \rho}{4(\beta-1)} \nabla^2 \psi(\rho q)
        \end{pmatrix} \geq 0,\quad \forall \rho>0,~q\in \mathbb{R}^4.
    \end{equation*}
    Denoting by $\psi_\rho(q)=\frac{\psi(\rho q)}{\rho}$, we find that the previous expression is nothing but
    \begin{equation*}
        \begin{pmatrix}
            \psi_\rho(q)& -\frac{\nabla^\top \psi_\rho(q)}{2}\\
            -\frac{\nabla \psi_\rho(q))}{2} & \frac{\beta }{4(\beta-1)} \nabla^2 \psi_\rho(q)
        \end{pmatrix} \geq 0,\quad \forall \rho>0,~q\in \mathbb{R}^4.
    \end{equation*}
    Furthermore, due to the homogeneity of $\psi$, we have that $\psi_\rho(q)=\rho^{\beta-1} \psi(q)$, and so
    \begin{equation*}
         \begin{pmatrix}
            \psi_\rho(q)& -\frac{\nabla^\top \psi_\rho(q)}{2}\\
            -\frac{\nabla \psi_\rho(q)}{2} & \frac{\beta }{4(\beta-1)} \nabla^2 \psi_\rho(q)
        \end{pmatrix}=\rho^{\beta-1} \begin{pmatrix}
            \psi(q)& -\frac{\nabla^\top \psi(q)}{2}\\
            -\frac{\nabla \psi( q)}{2} & \frac{\beta }{4(\beta-1)} \nabla^2 \psi(q)
        \end{pmatrix}.
    \end{equation*}
    The latter means that we need to prove that
    \begin{equation*}
        \begin{pmatrix}
            \psi(q)& -\frac{\nabla^\top \psi(q)}{2}\\
            -\frac{\nabla \psi( q)}{2} & \frac{\beta }{4(\beta-1)} \nabla^2 \psi(q)
        \end{pmatrix} \geq 0,\quad \forall q\in \mathbb{R}^4.
    \end{equation*}
    For every $(\xi,\chi) \in \mathbb{R}\times \mathbb{R}^4$ we have to prove that
    \begin{equation*}
    \begin{split}
        &(\xi,\chi)^\top \begin{pmatrix}
            \psi(q)& -\frac{\nabla^\top \psi(q)}{2}\\
            -\frac{\nabla \psi( q)}{2} & \frac{\beta }{4(\beta-1)} \nabla^2 \psi(q)
        \end{pmatrix} (\xi,\chi)\\
        =&\xi^2 \psi(q)-\xi \nabla \psi(q) \cdot \chi+\frac{\beta}{4(\beta-1)}\chi^\top \nabla^2 \psi(q) \chi \geq 0
    \end{split}
    \end{equation*}
    When $\psi(q)=0$ we have that $\nabla \psi(q)=0$, and the inequality follows from $\nabla^2 \psi(q)\geq 0$. Hence, assuming that $\psi(q)>0$ and optimizing with respect to $\xi$, we obtain an equivalent inequality
    \begin{equation*}
        -\frac{|\nabla\psi(q)\cdot \chi|^2}{4\psi(q)}+\frac{\beta}{4(\beta-1)} \chi^\top \nabla^2 \psi(q) \chi \geq 0,
    \end{equation*}
    or
    \begin{equation*}
        \chi^\top \left( \frac{\beta}{\beta-1} \psi(q) \nabla^2 \psi(q)-\nabla \psi(q) \otimes \nabla \psi(q) \right) \chi \geq 0.
    \end{equation*}
    Denoting by
    \begin{equation*}
        \tilde{q}=(-q_1^{-},q_2^+,-q_3^{-},q_4^{+}),\quad \text{for} \quad q=(q_1,q_2,q_3,q_4),
    \end{equation*}
    we have that
    \begin{equation*}
    \begin{split}
        \nabla \psi(q)=|\tilde{q}|^{\beta-2} \tilde{q},\quad \nabla^2 \psi(q)=(\beta-2)|\tilde{q}|^{\beta-4} \Tilde{q}\otimes \Tilde{q}+ |\Tilde{q}|^{\beta-2} M(q),
    \end{split}
    \end{equation*}
    where
    \begin{equation*}
        M(q)=\operatorname{diag}(\operatorname{Hvs}(-q_1),\operatorname{Hvs}(q_2),\operatorname{Hvs}(-q_3),\operatorname{Hvs}(q_4)),
    \end{equation*}
    and $\operatorname{Hvs}$ is the Heaviside step function. Hence we obtain that
    \begin{equation*}
        \begin{split}
        &\frac{\beta}{\beta-1} \psi(q) \nabla^2 \psi(q)-\nabla \psi(q) \otimes \nabla \psi(q)\\
        =&\frac{\beta-2}{\beta-1}|\tilde{q}|^{2\beta-4} \Tilde{q}\otimes \Tilde{q} + \frac{1}{\beta-1} |\Tilde{q}|^{2\beta-2} M(q)-|\tilde{q}|^{2\beta-4} \Tilde{q}\otimes \Tilde{q}\\
        =&\frac{|\Tilde{q}|^{2\beta-4}}{\beta-1} (|\Tilde{q}|^2 M(q)-\Tilde{q}\otimes \Tilde{q}),
        \end{split}
    \end{equation*}
    and therefore we need to prove that 
    \begin{equation*}
    \begin{split}
        &((q_1^-)^2+(q_2^+)^2+(q_3^-)^2+(q_4^+)^2)\\
        \times &(\chi_1^2 \operatorname{Hvs}(-q_1)+\chi_2^2 \operatorname{Hvs}(q_2)+\chi_3^2 \operatorname{Hvs}(-q_3)+\chi_4^2 \operatorname{Hvs}(q_4))\\
        \geq &  (-q_1^- \chi_1+q_2^+ \chi_2-q_3^- \chi_3+q_4^+ \chi_4)^2.
    \end{split}
    \end{equation*}
    The latter simply follows from the Cauchy-Schwarz inequality
    \begin{equation*}
        \sum_{i=1}^4 a_i^2 \cdot \sum_{i=1}^4 b_i^2 \geq \left(\sum_{i=1}^4 a_i b_i\right)^2 
    \end{equation*}
    applied to
    \begin{equation*}
    \begin{split}
        a_1=&-q_1^-,~a_2=q_2^+,~a_3=-q_3^-,~a_4=q_4^+,\\
        b_1=&\chi_1 \operatorname{Hvs}(-q_1),~b_2=\chi_2 \operatorname{Hvs}(q_2),~b_3=\chi_3 \operatorname{Hvs}(-q_3),~b_4=\chi_4 \operatorname{Hvs}(q_4),
    \end{split}
    \end{equation*}
    and taking into account that
    \begin{equation*}
       q^- \operatorname{Hvs}(-q)=q^-,\quad q^+ \operatorname{Hvs}(q)=q^+,\quad \operatorname{Hvs}(q)^2=\operatorname{Hvs}(q),\quad \forall q\in \mathbb{R}.
    \end{equation*}
    This completes the proof.
\end{proof}

\subsection{Properties of the discrete energy}

Here, we discuss key properties of the function $J$ defined in~\eqref{eq:J}.

\begin{lemma}\label{lma:J_min_eta}
    Assume that $\eta \geq 0$, and $F(t,x,\cdot)$, $G(x,\cdot)$ are convex and continuous for $\rho \geq 0$. Furthermore, for $(\lambda, y,z)\in \{0,1\}\times \mathcal{M} \times \mathcal{W}$ denote by
    \begin{equation}\label{eq:J_lyz}
        J_{\lambda,y,z}(\rho,w,\eta)=\lambda \left(\frac{|\rho|^2}{2}-\rho \cdot y  + \frac{|w|^2}{2}-w \cdot z \right) + J(\rho,w,\eta).
    \end{equation}
    Then the following statements are true.
    \begin{enumerate}
        \item The convex program
    \begin{equation}\label{eq:J_min_eta}
        \inf_{(1-\lambda)(A\rho+Bw)=0} J_{\lambda,y,z}(\rho,w,\eta)
    \end{equation}
    admits a minimizer for all $(\lambda, y,z)\in \{0,1\} \times \mathcal{M} \times \mathcal{W}$.
        \item When $\nu>0$ and $\lambda=0$ all minimizers of~\eqref{eq:J_min_eta} satisfy
        \begin{equation}\label{eq:rho_positive}
    \rho^{k+1}_{ij}>0,\quad \forall 0\leq i,j \leq N_h-1,~0 \leq k \leq N_T-1.
\end{equation}
    \end{enumerate}
\end{lemma}

\begin{remark}
The convex program~\eqref{eq:J_min_eta} is analogous to the problem $(P_{h,\Delta t})$ in~\cite{silva19} with the difference of having extra parameters $\lambda,y,z,\eta$. Hence, the proof techniques used in~\cite[Theorem 3.1]{silva19} and auxiliary lemmas extend to setting of~\cref{lma:J_min_eta}.    
\end{remark}

\begin{remark}
    The convex program~\eqref{eq:J_min_eta} is a shorthand for considering two programs
    \begin{equation*}
        \inf_{A\rho+Bw=0} J(\rho,w,\eta)
    \end{equation*}
    and
    \begin{eqnarray*}
        \inf_{\rho,w} \left\{\frac{|\rho|^2}{2}-\rho \cdot y  + \frac{|w|^2}{2}-w \cdot z+J(\rho,w,\eta) \right\}
    \end{eqnarray*}
    simultaneously. Both are important for our further analysis.
\end{remark}

\begin{proof}[Proof of~\cref{lma:J_min_eta}]

\begin{enumerate}
        \item We argue by the direct method of calculus of variations. Fix a triple $(\lambda,y,z) \in \{0,1\}\times \mathcal{M} \times \mathcal{W}$. According to~\cite[Lemma 3.1]{silva19} there exists $(\tilde{\rho},\tilde{w}) \in \mathcal{M}\times \mathcal{W}$ such that
        \begin{equation}\label{eq:tilde_rho>0}
            \tilde{\rho}^{k+1}_{ij}>0,\quad \Tilde{w}^k_{ij} \in \operatorname{int}(K),\quad \forall 0\leq i,j \leq N_h-1,~0\leq k \leq N_T-1,
        \end{equation}
        and
        \begin{equation}\label{eq:tilde_rho_sol}
            A\tilde{\rho}+B \tilde{w}=0,\quad \tilde{\rho}^0_{ij}=(\rho_0)_{ij},\quad \forall 1\leq i,j \leq N_h-1.
        \end{equation}
        Hence, $J_{\lambda,y,z}(\tilde{\rho},\tilde{w},\eta)<\infty$, and so
        \begin{equation*}
            \inf_{(1-\lambda)(A\rho+Bw)=0} J_{\lambda,y,z}(\rho,w,\eta)\leq J_{\lambda,y,z}(\tilde{\rho},\tilde{w},\eta)<\infty.
        \end{equation*}
        \begin{enumerate}
            \item When $\lambda=0$, for every $(\rho,w)$ such that $\rho\geq0$, $\rho^0=\rho_0$, and $A\rho+B w=0$ one has that~\cite[Section 3]{silva19}
        \begin{equation*}
            \sum_{0\leq i,j \leq N_h-1} \rho^k_{ij}=\frac{1}{h^2},\quad \forall 0\leq k \leq N_T,
        \end{equation*}
        and so
        \begin{equation}\label{eq:rho_unibounds}
            0\leq \rho^k_{ij}\leq \frac{1}{h^2},\quad \forall 0\leq i,j\leq N_h-1,~0\leq k \leq N_T.
        \end{equation}
        Thus, taking into account the continuity of $F,G$ and the fact that $E_h \geq 0$, we find that
        \begin{equation*}
            \inf_{(1-\lambda)(A\rho+Bw)=0} J_{\lambda,y,z}(\rho,w,\eta)=\inf_{A\rho+Bw=0} J(\rho,w,\eta)>-\infty.
        \end{equation*}
            \item When $\lambda=1$, the convexity of $F(t,x,\cdot)$ and $G(t,x,\cdot)$ yields
            \begin{equation}\label{eq:F_G_lower_bound}
            \begin{split}
                F(t_k,x_{ij},\rho^{k+1}_{ij}) \geq & F(t_k,x_{ij},1)+(\rho^{k+1}_{ij}-1) d^{k+1}_{ij}\\
                G(x_{ij},\rho^{N_T}_{ij}) \geq & G(x_{ij},1)+(\rho^{N_T}_{ij}-1) r_{ij}
            \end{split}
            \end{equation}
            for some $d^{k+1}_{ij} \in \partial_\rho F(t_k,x_{ij},1)$ and $r_{ij} \in \partial_\rho G(x_{ij},1)$. Thus, applying $E_h \geq 0$ in combination with the Cauchy-Schwarz inequality, we obtain
            \begin{equation*}
                J_{1,y,z}(\rho,w,\eta)\geq \frac{|\rho|^2}{2}-\rho \cdot \tilde{y}+\frac{|w|^2}{2}-w \cdot z -c \geq -\frac{|\tilde{y}|^2}{2}-\frac{|z|^2}{2}-c,
            \end{equation*}
            where
            \begin{equation}\label{eq:J_lower_bound_terms}
            \begin{split}
                \tilde{y}^{k+1}_{ij}=&y^{k+1}_{ij}-d^{k+1}_{ij},\quad 0\leq k \leq N_T-2\\
                \tilde{y}^{N_T}_{ij}=&y^{N_T}_{ij}-d^{N_T}_{ij}-r_{ij},\\
                c=&\sum\limits_{\substack{0\leq i,j \leq N_h-1\\ 0\leq k\leq N_T-1}} (d^{k+1}_{ij}-F(t_k,x_{ij},1))\\
                &+\frac{1}{\Delta t}\sum\limits_{0\leq i,j \leq N_h-1} (r^{N_T}_{ij}-G(x_{ij},1)).
            \end{split}
            \end{equation}
            Hence
            \begin{equation*}
                \inf_{(1-\lambda)(A\rho+Bw)=0} J_{\lambda,y,z}(\rho,w,\eta)=\inf_{\rho,w} J_1(\rho,w,\eta)>-\infty.
            \end{equation*}
        \end{enumerate}

        \vskip 0.5cm
        
        From (a), (b) above, we find that the infimum in~\eqref{eq:J_min_eta} is always finite. Let $(\rho_n,w_n)$ be a minimizing sequence for~\eqref{eq:J_min_eta}. Then we have that
        \begin{equation*}
            J_{\lambda,y,z}(\rho_n,w_n,\eta) \leq C,\quad \forall n \geq 1,
        \end{equation*}
        for some $C>0$. Our goal is to show that the sequence $\left\{(\rho_n,w_n) \right\}$ is bounded and extract a convergent subsequence. Again, let us discuss two cases.
        \begin{enumerate}
            \item When $\lambda=0$, taking into account~\eqref{eq:rho_unibounds} for $\rho_n$ and the continuity of $F,G$, we have that
        \begin{equation}\label{eq:E_h_unibound}
        \begin{split}
          E_h(t_k,x_{ij},(\rho_n)^{k+1}_{ij},(w_n)^k_{ij},\eta^{k+1}_{ij}) \leq \Tilde{C},
        \end{split}
        \end{equation}
        for all $n\geq 1$, $0\leq i,j \leq N_h-1$, $0\leq k \leq N_T-1$, and some $\Tilde{C}>0$. Hence if $(\rho_n)^{k+1}_{ij}>0$ then
        \begin{equation}\label{eq:w_unibound}
            |(w_n)^k_{ij}|\leq \left(\frac{\beta'}{\epsilon^{\alpha(\beta'-1)}}\left(\frac{1}{h^2}\right)^{\beta'-1}\Tilde{C}\right)^{\frac{1}{\beta}}.
        \end{equation}
        If $(\rho_n)^{k+1}_{ij}=0$ then we necessarily have that $(w_n)^k_{ij}=0$. Thus, in any case~\eqref{eq:w_unibound} holds for all $i,j,k$.
        \item Let $\lambda=1$, and $\tilde{y},c,d,r$ be as in~\eqref{eq:F_G_lower_bound} and~\eqref{eq:J_lower_bound_terms}. Then we have that
        \begin{equation*}
            \begin{split}
                C \geq & J_{1,y,z}(\rho_n,w_n,\eta) \geq  \frac{|\rho_n|^2}{2}-\rho_n \cdot \Tilde{y}+\frac{|w_n|^2}{2}-w_n\cdot z -c \\
                \geq& \frac{|\rho_n|^2}{4}+\frac{|w_n|^2}{4}-|\tilde{y}|^2-|z|^2-c, 
            \end{split}
        \end{equation*}
        and so
        \begin{equation*}
            \frac{|\rho_n|^2}{4}+\frac{|w_n|^2}{4} \leq |\tilde{y}|^2+|z|^2+c+C,\quad \forall n \geq 1.
        \end{equation*}
        \end{enumerate}
        
        \vskip 0.5cm
        
        Thus, for both cases $\lambda=0$ and $\lambda=1$ the minimizing sequence $\left\{(\rho^n,w^n)\right\}$ is precompact, and there exists a (subsequential) limit $(\rho^*,w^*) \in \mathcal{M}\times \mathcal{W}$. Since $\rho_n\geq 0$, $(\rho_n)^0=\rho_0$ and $(1-\lambda)(A\rho_n+Bw_n)=0$, we have that 
        \begin{equation*}
            \rho^*\geq 0,\quad (\rho^*)^0=\rho_0,\quad (1-\lambda)(A\rho^*+B w^*)=0.
        \end{equation*}
        Furthermore, taking into account the lower semicontinuity of $E_h(t,x,\cdot,\cdot,\eta)$ (see~\cref{rmk:perspective_function}) and the continuity of $F,G$ we obtain that
        \begin{equation*}
            J_{\lambda,y,z}(\rho^*,w^*,\eta) \leq \liminf_{n\to \infty} J_{\lambda,y,z}(\rho_n,w_n,\eta),
        \end{equation*}
        and so $(\rho^*,w^*)$ is a minimizer for~\eqref{eq:J_min_eta}.
        \item See the proof of the analogous statement in~\cite[Lemma 3.2]{silva19}.
    \end{enumerate}
\end{proof}

Denote by
\begin{equation}\label{eq:P}
    \mathcal{P}_0=\left\{\eta \in \mathcal{M}~:~\eta \geq 0,~\eta^0=\rho_0,~\sum_{0\leq i,j \leq N_h-1} \eta^k_{ij}=\frac{1}{h^2},~\forall 1\leq k \leq N_T\right\},
\end{equation}
and
\begin{equation}\label{eq:P1}
    \mathcal{P}_1=\left\{\eta \in \mathcal{M}~:~\eta \geq 0,~\eta^0=\rho_0,~|\eta|^2\leq c_1\right\},
\end{equation}
where
\begin{equation*}
\begin{split}
    c_1=&4\Bigg(|\tilde{y}|^2+|z|^2+c+\frac{|\bar{\rho}|^2}{2}-\bar{\rho}\cdot y+\sum\limits_{\substack{0\leq i,j \leq N_h-1\\ 0\leq k\leq N_T-1}} F(t_k,x_{ij},1)\\
    &+\frac{1}{\Delta t}\sum\limits_{0\leq i,j \leq N_h-1} G(x_{ij}, 1)\Bigg),\\
    \bar{\rho}^0=&\rho_0,\quad \bar{\rho}^{k+1}_{ij}=1,~\forall 0\leq k \leq N_T-1,
\end{split}
\end{equation*}
and $\tilde{y},c,d,r$ are as in~\eqref{eq:F_G_lower_bound} and~\eqref{eq:J_lower_bound_terms}.

Furthermore, for $\eta \geq 0$ and $(\lambda,y,z) \in \{0,1\}\times \mathcal{M} \times \mathcal{W}$ denote by
\begin{equation}\label{eq:S-eta}
    \mathcal{S}_{\lambda,y,z}(\eta)=\left\{\rho^*~:~\exists w^*\in \mathcal{W}~\text{s.t.}~(\rho^*,w^*)\in \underset{(1-\lambda)(A\rho+B w)=0}{\operatorname{argmin}}J_{\lambda,y,z}(\rho,w,\eta)\right\}.
\end{equation}

\begin{lemma}\label{lma:fixed_point}
    Assume that $F(t,x,\cdot)$ and $G(x,\cdot)$ are convex and continuous for $\rho\geq 0$, and $(\lambda,y,z) \in \{0,1\}\times \mathcal{M} \times \mathcal{W}$. 
    Then the set valued map $\mathcal{S}_{\lambda,y,z}:\mathcal{P}_\lambda \rightrightarrows \mathcal{P}_\lambda$ admits a fixed point; that is, there exists $\rho^* \in \mathcal{P}_\lambda$ such that $\rho^* \in \mathcal{S}_{\lambda,y,z}(\rho^*)$.
\end{lemma}
\begin{proof}
    Our strategy is to apply Kakutani's fixed point theorem~\cite{kakutani41generalization}. Hence, fix a triple $(\lambda,y,z) \in \{0,1\}\times \mathcal{M} \times \mathcal{W}$.
    \begin{enumerate}
        \item \cref{lma:J_min_eta} yields that $\mathcal{S}_{\lambda,y,z}(\eta)\neq \emptyset$ for all $\eta\geq 0$.
        \begin{enumerate}
            \item When $\lambda=0$, we have that for all $\rho^* \in \mathcal{S}_{\lambda,y,z}(\eta)$ there exists $w^* \in \mathcal{W}$ such that
            \begin{equation*}
                (\rho^*,w^*)\in \underset{A \rho +Bw =0}{\operatorname{argmin}} J(\rho,w,\eta). 
            \end{equation*}
            Hence, $\rho^* \in \mathcal{P}_0$, and so $\mathcal{S}_{0,y,z}(\eta) \subset \mathcal{P}_0$ for all $\eta \in \mathcal{P}_0$.
            \item When $\lambda=1$, we have that for all $\rho^* \in \mathcal{S}_{\lambda,y,z}(\eta)$ there exists $w^* \in \mathcal{W}$ such that
            \begin{equation*}
                (\rho^*,w^*)\in \underset{\rho,w}{\operatorname{argmin}}~J_{1,y,z}(\rho,w,\eta),
            \end{equation*}
            and so
            \begin{equation*}
                J_{1,y,z}(\rho^*,w^*,\eta) \leq J_{1,y,z}(\bar{\rho},0,\eta).
            \end{equation*}
            Hence, applying lower bounds on $J_{1,y,z}$ as in~\cref{lma:J_min_eta} we obtain
            \begin{equation*}
                \frac{|\rho^*|^2}{4}+\frac{|w^*|^2}{4}\leq |\tilde{y}|^2+|z|^2+c+J_{1,y,z}(\bar{\rho},0,\eta).
            \end{equation*}
            Noting that
            \begin{equation*}
            \begin{split}
                &J_1(\bar{\rho},0,\eta)\\
                =&\frac{|\bar{\rho}|^2}{2}-\bar{\rho}\cdot y+\sum\limits_{\substack{0\leq i,j \leq N_h-1\\ 0\leq k\leq N_T-1}} F(t_k,x_{ij},1)
    +\frac{1}{\Delta t}\sum\limits_{0\leq i,j \leq N_h-1} G(x_{ij}, 1),
            \end{split}
            \end{equation*}
            we conclude that $\rho^* \in \mathcal{P}_1$, and so $\mathcal{S}_{1,y,z}(\eta) \subset \mathcal{P}_1$ for all $\eta \in \mathcal{P}_1$.
        \end{enumerate}

        \vskip 0.5cm

        Summarizing (a), (b) above, we conclude that $\mathcal{S}_{\lambda,y,z}(\eta) \subset \mathcal{P}_\lambda$ for all $\eta \in \mathcal{P}_\lambda$. Additionally, the convexity and continuity of $F(t,x,\cdot)$ and $G(x,\cdot)$ yield that $J_{\lambda,y,z}(\cdot,\cdot,\eta)$ is convex and lower semicontinuous, and so
        \begin{equation*}
            \underset{(1-\lambda)(A\rho+B w)=0}{\operatorname{argmin}} J_{\lambda,y,z}(\rho,w,\eta)
        \end{equation*}
        is a closed convex set. Thus, $\mathcal{S}_{\lambda,y,z}(\eta)$ is also closed and convex.
        \item Assume that $\{(\eta_n,\rho_n)\} \subset \mathcal{P}_\lambda \times \mathcal{P}_\lambda$ is such that $\rho_n \in \mathcal{S}_{\lambda,y,z}(\eta_n)$ for all $n\geq 1$, and
        \begin{equation*}
            \lim\limits_{n \to \infty}(\eta_n,\rho_n) =(\hat{\eta},\hat{\rho}).
        \end{equation*}
        Since $\mathcal{P}_\lambda$ is compact, we have that $\hat{\eta} \in \mathcal{P}_\lambda$. Furthermore, let $w_n\in \mathcal{W}$ be such that 
        \begin{equation*}
            (\rho_n,w_n) \in \underset{(1-\lambda)(A\rho +B w)=0}{\operatorname{argmin}} J_{\lambda,y,z}(\rho,w,\eta_n).
        \end{equation*}
        Fix a $(\tilde{\rho},\tilde{w})\in \mathcal{M} \times \mathcal{W}$ such that~\eqref{eq:tilde_rho>0} and \eqref{eq:tilde_rho_sol} hold. Then we have that
        \begin{equation*}
            J_{\lambda,y,z}(\rho_n,w_n,\eta_n)\leq J_{\lambda,y,z}(\tilde{\rho},\tilde{w},\eta_n),\quad \forall n\geq 1.
        \end{equation*}
        Furthermore, we have that
        \begin{equation*}
            \lim_{n\to \infty} J_{\lambda,y,z}(\tilde{\rho},\tilde{w},\eta_n)=J_{\lambda,y,z}(\tilde{\rho},\tilde{w},\hat{\eta})<\infty.
        \end{equation*}
        Hence, there exists $C>0$ such that
        \begin{equation*}
            -\lambda \left(\frac{|y|^2}{2}+\frac{|z|^2}{2} \right)+ J(\rho_n,w_n,\eta_n)\leq J_{\lambda,y,z}(\rho_n,w_n,\eta_n) \leq C,\quad \forall n\geq 1,
        \end{equation*}
        and so
        \begin{equation*}
            J(\rho_n,w_n,\eta_n) \leq C+\lambda \left(\frac{|y|^2}{2}+\frac{|z|^2}{2} \right),\quad \forall n\geq 1.
        \end{equation*}
        Since $\{\rho_n\}$ is a bounded sequence and $F(t,x,\cdot)$, $G(x,\cdot)$ are continuous, we obtain that
        \begin{equation}\label{eq:E_h_unibound_2}
           0\leq  E_h(t_k,x_{ij},(\rho_n)^{k+1}_{ij},(w_n)^k_{ij},(\eta_n)^{k+1}_{ij}) \leq \Tilde{C},
        \end{equation}
        for $0\leq i,j \leq N_h-1$, $0\leq k \leq N_T-1$, and some $\Tilde{C}>0$. Next, denote by
        \begin{equation*}
            \ell(v_1,v_2,v_3,v_4)=\begin{cases}
                \frac{(v_1^2+v_2^2+v_3^2+v_4^2)^{\beta'/2}}{\beta'},\quad (v_1,v_2,v_3,v_4) \in K,\\
                \infty,\quad \text{otherwise}.
            \end{cases}
        \end{equation*}
        The corresponding recession function is then
        \begin{equation*}
            \tilde{\ell}(\rho,w_1,w_2,w_3,w_4) =\begin{cases}
                \frac{(w_1^2+w_2^2+w_3^2+w_4^2)^{\beta'/2}}{\beta' \rho^{\beta'-1}},\quad (w_1,w_2,w_3,w_4) \in K,\\
                0,\quad (\rho,w_1,w_2,w_3,w_4)=(0,0,0,0,0),\\
                \infty,\quad \text{otherwise}.
            \end{cases}
        \end{equation*}
        Note that for $\eta\geq 0$ we have that 
        \begin{equation}\label{eq:ell_L}
            \begin{split}
                L_h(t,x,v_1,v_2,v_3,v_4,\eta)=&(\eta+\epsilon)^{\alpha(\beta'-1)} \ell(v_1,v_2,v_3,v_4),\\
                E_h(t,x,w_1,w_2,w_3,w_4,\eta)=&(\eta+\epsilon)^{\alpha(\beta'-1)} \tilde{\ell}(\rho,w_1,w_2,w_3,w_4).
            \end{split}
        \end{equation}
        Taking into account $\eta_n \geq 0$ and~\eqref{eq:E_h_unibound_2} we obtain
        \begin{equation}
            0\leq \tilde{\ell}((\rho_n)^{k+1}_{ij},(w_n)^k_{ij})\leq \frac{\Tilde{C}}{\epsilon^{\alpha(\beta'-1)}}.
        \end{equation}
        In particular, we find that~\eqref{eq:w_unibound} holds. Thus $\{(\rho_n,w_n)\}$ is precompact, and, possibly through a subsequence, we have that
        \begin{equation*}
            \lim_{n \to \infty} (\rho_n,w_n)=(\hat{\rho},\hat{w}).
        \end{equation*}
        Hence, we have that
        \begin{equation*}
        \begin{split}
            \hat{\rho}^0=&\lim_{n \to \infty} (\rho_n)^0=\rho_0, \quad \hat{\rho}=\lim_{n \to \infty} \rho_n \geq 0,\\
            (1-\lambda)(A \hat{\rho}+ B \hat{w})=&\lim_{n \to \infty} (1-\lambda)(A\rho_n+B w_n)=0.
        \end{split}
        \end{equation*}
        Additionally, using the lower semicontinuity of $\Tilde{\ell}$, we find that
        \begin{equation*}
            \begin{split}
                &\liminf_{n\to \infty} E_h(t_k,x_{ij},(\rho_n)^{k+1}_{ij},(w_n)^k_{ij},(\eta_n)^{k+1}_{ij}) \\
                \geq& \liminf_{n \to \infty} ((\eta_n)^{k+1}_{ij}+\epsilon)^{\alpha(\beta'-1)} \cdot \liminf_{n \to \infty} \tilde{\ell}((\rho_n)^{k+1}_{ij},(w_n)^k_{ij})\\
                \geq & (\hat{\eta}^{k+1}_{ij}+\epsilon)^{\alpha(\beta'-1)} \tilde{\ell}(\hat{\rho}^{k+1}_{ij},\hat{w}^{k}_{ij})=E_h(t_k,x_{ij},\hat{\rho}^{k+1}_{ij},\hat{w}^{k}_{ij},\hat{\eta}^{k+1}_{ij}),
            \end{split}
        \end{equation*}
        for all $0\leq i,j \leq N_h-1$, and $0\leq l \leq N_T-1$. But then we have that
        \begin{equation*}
            \begin{split}
                &\liminf_{n \to \infty} J_{\lambda,y,z}(\rho_n,w_n,\eta_n) \\
                \geq& \liminf_{n \to \infty} \lambda \left(\frac{|\rho_n|^2}{2}-\rho_n \cdot y+\frac{|w_n|^2}{2}-w_n \cdot z \right)\\
                &+ \sum_{\substack{0\leq i,j \leq N_h-1\\0\leq k \leq N_T-1}} \liminf_{n\to \infty} E_h(t_k,x_{ij},(\rho_n)^{k+1}_{ij},(w_n)^k_{ij},(\eta_n)^{k+1}_{ij})\\
                &+\sum_{\substack{0\leq i,j \leq N_h-1\\0\leq k \leq N_T-1}} \liminf_{n\to \infty} F(t_k,x_{ij},(\rho_n)^{k+1}_{ij})\\
                &+\frac{1}{\Delta t} \sum_{0\leq i,j \leq N_h-1} \liminf_{n\to \infty} G(x_{ij},(\rho_n)^{N_T}_{ij})\\
                \geq& \lambda \left(\frac{|\hat{\rho}|^2}{2}-\hat{\rho} \cdot y+\frac{|\hat{w}|^2}{2}-\hat{w} \cdot z \right)\\
                &+\sum_{\substack{0\leq i,j \leq N_h-1\\0\leq k \leq N_T-1}} E_h(t_k,x_{ij},\hat{\rho}^{k+1}_{ij},\hat{w}^k_{ij},\hat{\eta}^{k+1}_{ij})\\
                &+\sum_{\substack{0\leq i,j \leq N_h-1\\0\leq k \leq N_T-1}} F(t_k,x_{ij},\hat{\rho}^{k+1}_{ij})+\frac{1}{\Delta t} \sum_{0\leq i,j \leq N_h-1} G(x_{ij},\hat{\rho}^{N_T}_{ij})\\
                =&J_{\lambda,y,z}(\hat{\rho},\hat{w},\hat{\eta}).
            \end{split}
        \end{equation*}
        Now fix $(\rho,w) \in \mathcal{M}\times \mathcal{W}$ be such that $\rho^0=\rho_0$, and $(1-\lambda)(A\rho + Bw) = 0$. Our goal is to prove that
        \begin{equation*}
            J_{\lambda,y,z}(\hat{\rho},\hat{w},\hat{\eta})\leq J_{\lambda,y,z}(\rho,w,\hat{\eta}),
        \end{equation*}
        so that $\hat{\rho} \in \mathcal{S}_{\lambda,y,z}(\hat{\eta})$.
        
        If $J_{\lambda,y,z}(\rho,w,\hat{\eta})=\infty$, the inequality is trivially true. Hence, assume that $J_{\lambda,y,z}(\rho,w,\hat{\eta})<\infty$. Note that this implies that $J_{\lambda,y,z}(\rho,w,\eta)<\infty$ for all $\eta\geq 0$. We have that
        \begin{equation*}
            J_{\lambda,y,z}(\rho_n,w_n,\eta_n)\leq J_{\lambda,y,z}(\rho,w,\eta_n),\quad \forall n\geq 1,
        \end{equation*}
        but then
        \begin{equation*}
            \begin{split}
                J_{\lambda,y,z}(\hat{\rho},\hat{w},\hat{\eta})\leq& \liminf_{n\to \infty} J_{\lambda,y,z}(\rho_n,w_n,\eta_n) \leq \liminf_{n\to \infty} J_{\lambda,y,z}(\rho,w,\eta_n)\\
                =&\lim_{n\to \infty} J_{\lambda,y,z}(\rho,w,\eta_n)=J_{\lambda,y,z}(\rho,w,\hat{\eta}).
            \end{split}
        \end{equation*}
    \end{enumerate}
    Hence, $\mathcal{S}_{\lambda,y,z}$ maps all $\eta \in \mathcal{P}_\lambda$ to nonempty closed convex subsets of $\mathcal{P}_\lambda$ and has a closed graph. Therefore, Kakutani's theorem~\cite{kakutani41generalization} implies that $\mathcal{S}_{\lambda,y,z}$ has a fixed point.
\end{proof}

\begin{corollary}\label{crl:K_max_mon}
Assume that $F(t,x,\cdot)$ and $G(x,\cdot)$ are convex and continuous for $\rho \geq 0$, and for $(\rho,w)\in \mathcal{M}\times \mathcal{W}$ denote by
\begin{equation}\label{eq:K}
    \mathcal{K}(\rho,w)=\begin{cases}
        \partial_{(\rho,w)} J(\rho,w,\eta)\Big|_{\eta=\rho},\quad\rho \geq 0,\\
        \emptyset,\quad \text{otherwise}.
    \end{cases}
\end{equation}
Then $\mathcal{K}$ is maximally monotone.
\end{corollary}

\vskip 0.5cm

\begin{remark}
Note that in~\eqref{eq:K} we first differentiate $J$ with respect to $(\rho,w)$ and then plug in $\eta=\rho$. This should not be confused with $\partial_{(\rho,w)} J(\rho,w,\rho)$.
\end{remark}

\vskip 0.5cm

\begin{proof}[Proof of~\cref{crl:K_max_mon}]
    By Minty's theorem~\cite[Theorem 3.5.8]{aubin90set} we have to prove that $\mathcal{K}$ is monotone, and $I+\mathcal{K}$ is surjective. We start with the latter.
    \begin{enumerate}
        \item \textit{Surjectivity.} Assume that $(y,z)\in \mathcal{M}\times \mathcal{W}$ are arbitrary. By~\cref{lma:fixed_point} we have that there exists $\rho^* \in \mathcal{P}_1$ such that $\rho^* \in \mathcal{S}_{1,y,z}(\rho^*)$. Hence, there exists $w^* \in \mathcal{W}$ such that
    \begin{equation*}
        (\rho^*,w^*) \in \underset{\rho,w}{\operatorname{argmin}}~J_{1,y,z}(\rho,w,\rho^*).
    \end{equation*}
    Consequently, we have that
    \begin{equation*}
        (0,0) \in \partial_{(\rho,w)} J_{1,y,z}(\rho^*,w^*,\eta)\Big|_{\eta=\rho^*}.
    \end{equation*}
    Furthermore, we have that
    \begin{equation*}
        J_{1,y,z}(\rho,w,\eta)=\frac{|\rho|^2}{2}-\rho \cdot y+\frac{|w|^2}{2}-w \cdot z+J(\rho,w,\eta),
    \end{equation*}
    and
    \begin{equation*}
        (\rho,w) \mapsto \frac{|\rho|^2}{2}-\rho \cdot y+\frac{|w|^2}{2}-w \cdot z
    \end{equation*}
    is continuously differentiable. Hence, we obtain
    \begin{equation*}
        \partial_{\rho,w} J_{1,y,z}(\rho,w,\eta)=(\rho, w)- (y, z) + \partial_{(\rho,w)}J(\rho,w,\eta),
    \end{equation*}
    and so
    \begin{equation*}
        (0,0) \in (\rho^*,w^*)+\mathcal{K}(\rho^*,w^*)-(y,z) \Longleftrightarrow (y,z)\in (\rho^*,w^*)+\mathcal{K}(\rho^*,w^*).
    \end{equation*}
        \item \textit{Monotonicity.} The convexity of $F(t,x,\cdot)$ and $G(x,\cdot)$ and the separable structure of $J$ at the grid points yield that it is sufficient to prove that
        \begin{equation*}
            e_h(t,x,\rho,w)=\begin{cases}
                \partial_{\rho,w} E_h (t,x,\rho,w,\eta)\Big|_{\eta=\rho},\quad \rho \geq 0\\
                \emptyset,\quad \text{otherwise}
            \end{cases}
        \end{equation*}
        is monotone. Applying~\cite[Proposition 2.3]{combettes18perpsective}, we obtain
        \begin{equation}\label{eq:e_h_expanded}
        \begin{split}
            &e_h(t,x,\rho,w)\\
            =&\begin{cases}
                \left\{ \left(-H_h(t,x,q,\rho),-q\right)~:~q\in -\partial_v L_h(t,x,\frac{w}{\rho},\rho) \right\},\quad \rho>0,\\
                \left\{ (\mu,-q)~:~\mu+H_h(t,x,q,\rho)\leq 0\right\},\quad (\rho,w)=(0,0),\\
                \emptyset, \quad \text{otherwise}.
            \end{cases}
        \end{split}
        \end{equation}
        Now assume that
        \begin{equation*}
            (y_i,z_i) \in e_h(t,x,\rho_i,w_i),\quad i=1,2.
        \end{equation*}
        We have to show that
        \begin{equation*}
            (y_2-y_1) (\rho_2-\rho_1)+(z_2-z_1)\cdot (w_2-w_1) \geq 0.
        \end{equation*}
        Since the case $(\rho_1,w_1)=(\rho_2,w_2)$ is trivial, we assume that $(\rho_1,w_1)\neq (\rho_2,w_2)$. Hence, up to swapping $(\rho_1,w_1)$ and $(\rho_2,w_2)$, there are two possibilities:
        \begin{enumerate}
            \item $\rho_1,\rho_2>0$. In this case,~\eqref{eq:e_h_expanded} yields that there exist $q_1,q_2$ such that
            \begin{equation*}
            y_i=-H_h(t,x,q_i,\rho_i),~z_i=-q_i,~q_i \in -\partial_v L_h(t,x,\frac{w_i}{\rho_i},\rho_i),\quad i=1,2.
            \end{equation*}
            But then the Legendre duality yields that
            \begin{equation*}
                w_i=-\rho_i \nabla_q H_h(t,x,q_i,\rho_i),\quad i=1,2,
            \end{equation*}
            and so we have to prove that
            \begin{equation*}
            \begin{split}
                &(-H_h(t,x,q_2,\rho_2)+H_h(t,x,q_1,\rho_1))(\rho_2-\rho_1)\\
                &+(\rho_2 \nabla_q H_h(t,x,q_2,\rho_2)-\rho_1 \nabla_q H_h(t,x,q_1,\rho_1))\cdot(q_2-q_1) \geq 0.
            \end{split}
            \end{equation*}
            This inequality is equivalent to the monotonicity of the map
            \begin{equation}\label{eq:H_h_mon_map}
                (\rho,q) \mapsto (-H_h(t,x,q,\rho),\rho \nabla_q H_h(t,x,\rho,q)),\quad \rho \geq 0,~q\in \mathbb{R}^4.
            \end{equation}
            Since the map is continuously differentiable, its monotinicity is equivalent to the positive definiteness of the symmetric part of its Jacobian. Note that the Jacobian is given by
            \begin{equation*}
                \begin{pmatrix}
                    -\partial_\rho H_h & -\nabla_q^\top H_h\\
                    \nabla_q H_h+ \rho \nabla_q \partial_\rho H_h & \rho \nabla^2_q H_h
                \end{pmatrix},
            \end{equation*}
            and so its symmetric part is
            \begin{equation*}
                \begin{pmatrix}
                    -\partial_\rho H_h & \frac{1}{2}\rho \nabla_q^\top \partial_\rho H_h\\
                    \frac{1}{2}\rho \nabla_q \partial_\rho H_h & \rho \nabla^2_q H_h
                \end{pmatrix}.
            \end{equation*}
            Thus, the monotonicity of~\eqref{eq:H_h_mon_map} is equivalent to the Lasry-Lions monotonicity condition of $H_h$, which is proven in Lemma~\eqref{lma:LL_discrete}.
            \item $\rho_1=0,~\rho_2>0$. In this case,~\eqref{eq:e_h_expanded} yields
            \begin{equation*}
                w_1=0,~y_1\leq -H_h(t,x,q_1,0),~z_1=-q_1,
            \end{equation*}
            and
            \begin{equation*}
               y_2=-H_h(t,x,q_2,\rho_2),~z_2=-q_2,~q_2 \in -\partial_v L_h(t,x,\frac{w_2}{\rho_2},\rho_2).
            \end{equation*}
            Hence, we have that
            \begin{equation*}
                w_i=-\rho_i \nabla_q H_h(t,x,q_i,\rho_i),\quad i=1,2,
            \end{equation*}
            and we need to show that
            \begin{equation*}
            \begin{split}
                &(-H_h(t,x,q_2,\rho_2)-y_1)(\rho_2-\rho_1)\\
                &+(\rho_2 \nabla_q H_h(t,x,q_2,\rho_2)-\rho_1 \nabla_q H_h(t,x,q_1,\rho_1))\cdot(q_2-q_1) \geq 0.
            \end{split}
            \end{equation*}
            But inequalities $y_1\leq -H_h(t,x,q_1,0)$ and $\rho_2-\rho_1=\rho_2>0$ yield that
            \begin{equation*}
            \begin{split}
                &(-H_h(t,x,q_2,\rho_2)-y_1)(\rho_2-\rho_1)\\
                &+(\rho_2 \nabla_q H_h(t,x,q_2,\rho_2)-\rho_1 \nabla_q H_h(t,x,q_1,\rho_1))\cdot(q_2-q_1) \\
                \geq &(-H_h(t,x,q_2,\rho_2)+H_h(t,x,q_1,\rho_1))(\rho_2-\rho_1)\\
                &+(\rho_2 \nabla_q H_h(t,x,q_2,\rho_2)-\rho_1 \nabla_q H_h(t,x,q_1,\rho_1))\cdot(q_2-q_1)\geq 0,
            \end{split}
            \end{equation*}
            where the last inequality follows from the monotonicity of~\eqref{eq:H_h_mon_map} proven in (a).
        \end{enumerate}
    \end{enumerate}
\end{proof}

\subsection{First-order optimality conditions}

\begin{lemma}\label{lma:KKT}
Let $\eta \geq 0$, and assume that $F(t,x,\cdot)$ and $G(x,\cdot)$ are convex and continuous for $\rho \geq 0$. Then the following statements are true.
\begin{enumerate}
    \item $(\rho,w) \in \mathcal{M}\times \mathcal{W}$ is a minimizer for~\eqref{eq:J_min_eta} with $\lambda=0$ if and only if
\begin{equation}\label{eq:KKT_eta}
\begin{split}
    (A^*\phi,B^*\phi) \in \partial_{(\rho,w)} J(\rho,w,\eta),\quad A\rho+Bw=0,
\end{split}
\end{equation}
for some $\phi \in \mathcal{U}$.
    \item If~\eqref{eq:rho_positive} holds and $F(t,x,\cdot)$ and $G(x,\cdot)$ are differentiable for $\rho>0$, \eqref{eq:KKT_eta} is equivalent to
\begin{equation}\label{eq:HJB_KFP_discrete}
\begin{cases}
    -(D_t \phi_{ij})^k-\nu (\Delta_h \phi^k)_{ij}+H_h(x_{ij},\eta_{ij}^{k+1},[D_h \phi^k]_{ij})=f(t_k,x_{ij},\rho^{k+1}_{ij}),\\
    w^k_{ij}=-\rho^{k+1}_{ij}\nabla_q H_h(x_{ij},\eta^{k+1}_{ij},[D_h \phi^k]),\\
    (D_t \rho_{ij})^k-\nu (\Delta_h \rho^{k+1})_{ij}+(Bw)_{ij}^k=0,\\
    \rho^0_{ij}=(\rho_0)_{ij},~ \phi^{N_T}_{ij}=g(x_{ij},\rho^{N_T}_{ij}),\quad 0\leq i,j \leq N_h-1,~ 0\leq k \leq N_T-1.
\end{cases}
\end{equation}
    \item If~\eqref{eq:rho_positive} holds and $F(t,x,\cdot)$ and $G(x,\cdot)$ are differentiable for $\rho>0$,~\eqref{eq:mfg_discrete} is equivalent to
\begin{equation}\label{eq:inc_discrete_1}
    (A^*\phi,B^*\phi) \in \mathcal{K}(\rho,w),\quad A\rho+Bw=0,
\end{equation}
where $\mathcal{K}$ is defined in~\eqref{eq:K}.
\end{enumerate}
\end{lemma}
\begin{proof}
\begin{enumerate}
    \item Since~\eqref{eq:J_min_eta} is a convex program with affine constraints, KKT conditions are necessary and sufficient for optimality. To obtain the KKT conditions, we introduce a Lagrange multiplier, $\phi$, for the constraint $A \rho+ Bw=0$ and find
\begin{equation*}
\begin{split}
    \inf_{A\rho + B w=0} J(\rho,w,\eta)=&\inf_{(\rho,w)\in \mathcal{M}\times \mathcal{W}} \sup_{\phi\in \mathcal{U}} \left\{J(\rho,w,\eta)-\langle A \rho+B w, \phi \rangle \right\}\\
    =&\inf_{(\rho,w)\in \mathcal{M}\times \mathcal{W}} \sup_{\phi\in \mathcal{U}} \left\{J(\rho,w,\eta)-\langle \rho, A^* \phi \rangle - \langle w, B^*\phi \rangle \right\},
\end{split}
\end{equation*}
which yields~\eqref{eq:KKT_eta}.
\item Next, assume that~\eqref{eq:rho_positive} holds. In what follows, we use that for every $\rho>0$, $\eta\geq 0$, and $w\in K$ one has that~\cite[Proposition 2.3]{combettes18perpsective} 
\begin{equation*}
    \begin{split}
        \partial_{(\rho,w)} E_h(t,x,\rho,w,\eta)
        =&\left\{\left(-H_h(t,x,q,\eta),  -q \right)~:~q\in -\partial_v L_h(t,x,\frac{w}{\rho},\eta)\right\}.
    \end{split}
\end{equation*}
Hence~\eqref{eq:KKT_eta} is equivalent to
\begin{equation}\label{eq:KKT_1}
    \begin{split}
        -[D_h \phi^k]_{ij} \in& \partial_v L_h(t_k,x_{ij},\frac{w^{k}_{ij}}{\rho^{k+1}_{ij}},\eta^{k+1}_{ij}),\\
        -(D_t \phi_{ij})^k-\nu (\Delta_h \phi^k)_{ij}=&-H_h(t_k,x_{ij},[D_h\phi^k]_{ij},\eta^{k+1}_{ij})+f(t_k,x_{ij},\rho^{k+1}_{ij}),\\
    \frac{\phi^{N_T-1}_{ij}}{\Delta t}-\nu (\Delta_h \phi^{N_T-1})_{ij}=&-H_h(t_{N_T-1},x_{ij},[D_h \phi^{N_T-1}]_{ij},\eta^{N_T}_{ij})\\
    &+f(t_{N_T-1},x_{ij},\rho^{N_T}_{ij})+\frac{g(x_{ij},\rho^{N_T}_{ij})}{\Delta t}.
    \end{split}
\end{equation}
for all $0\leq i,j \leq N_h-1$ and $0 \leq k \leq N_T-1$. The second and third equalities can be combined in a system
\begin{equation}\label{eq:HJB_discrete}
\begin{cases}
    -(D_t \phi_{ij})^k-\nu (\Delta_h \phi^k)_{ij}+H_h(t_k,x_{ij},[D_h\phi^k]_{ij},\eta^{k+1}_{ij})=f(t_k,x_{ij},\rho^{k+1}_{ij}),\\
    \phi^{N_T}_{ij}=g(x_{ij},\rho^{N_T}_{ij}),\quad \forall 0\leq i,j \leq N_h-1,~0\leq k \leq N_T-1.
\end{cases}
\end{equation}
Next, using the convex duality again, we find that the first inclusion in~\eqref{eq:KKT_1} is equivalent to the equality
\begin{equation*}
    \frac{w^k_{ij}}{\rho^{k+1}_{ij}}=-\nabla_q H_h(t_k,x_{ij},[D_h \phi^k]_{ij},\eta^{k+1}_{ij}),
\end{equation*}
for all $0\leq i,j \leq N_h-1$ and $0 \leq k \leq N_T-1$, since $H_h(t,x,q,\eta)$ is continuously differentiable in $q$ for $\eta\geq 0$. Combining this previous identity with $A\rho+Bw=0$ we obtain
\begin{equation}\label{eq:KFP_discrete}
    \begin{cases}
    (D_t \rho_{ij})^k-\nu (\Delta_h \rho^{k+1})_{ij}+(Bw)_{ij}^k=0,\\
    w^k_{ij}=-\rho^{k+1}_{ij}\nabla_q H_h(x_{ij},\eta^{k+1}_{ij},[D_h \phi^k]),\\
    \rho^0_{ij}=\rho_{0,ij},\quad 0\leq i,j \leq N_h-1,~ 0\leq k \leq N_T-1.
\end{cases}
\end{equation}
Finally, combining~\eqref{eq:HJB_discrete} and~\eqref{eq:KFP_discrete} one obtains~\eqref{eq:HJB_KFP_discrete}.
\item This follows from part 2 and the definition of $\mathcal{K}$.
\end{enumerate}
\end{proof}

\begin{corollary}\label{crl:existence}
Assume that $\nu>0$, and $F(t,x,\cdot)$ and $G(x,\cdot)$ are convex and continuous for $\rho \geq 0$ and differentiable for $\rho>0$. Then~\eqref{eq:mfg_discrete} admits a solution.    
\end{corollary}
\begin{proof}
    By~\cref{lma:fixed_point} we have that there exist $(\rho^*,w^*)$ such that
    \begin{equation*}
        (\rho^*,w^*) \in \underset{A \rho +Bw =0}{\operatorname{argmin}} J(\rho,w,\rho^*).
    \end{equation*}
    Furthermore, by~\cref{lma:J_min_eta} we have that
    \begin{equation*}
        (\rho^*)^{k+1}_{ij}>0,\quad 0\leq k \leq N_T-1,~0\leq i,j \leq N_h-1.
    \end{equation*}
    Hence,~\cref{lma:KKT} yields $\phi^*$ such that $(\rho^*,w^*,\phi^*)$ solves~\eqref{eq:mfg_discrete}.
\end{proof}

\section{A pair of primal-dual monotone inclusions}

The formulation~\eqref{eq:inc_discrete_1} is the basis of our computational method. We first recall monotone inclusion version of PDHG. Following~\cite{vu13}, assume that $\mathcal{H}_1, \mathcal{H}_2$ are Hilbert spaces, and $M:\mathcal{H}_1 \to 2^{\mathcal{H}_1}$, $N:\mathcal{H}_2 \to 2^{\mathcal{H}_2}$ are maximally monotone operators, and $C:\mathcal{H}_1 \to \mathcal{H}_2$ is a nonzero bounded linear operator. Now consider the following pair or primal-dual monotone inclusions
\begin{equation}\label{eq:inc_PDHG_gen}
    \begin{split}
        \text{find}~x\in \mathcal{H}_1~\text{s.t.}~0\in Mx+C^*(N(Cx))&\quad \text{(P)}\\
        \text{find}~y\in \mathcal{H}_2~\text{s.t.}~y\in N(Cx),~-C^*y\in Mx,~\text{for some}~x\in \mathcal{H}_1&\quad \text{(D)} 
    \end{split}
\end{equation}
When $M,N$ are subdifferentials of proper convex lower semicontinuous functions; that is, $M=\partial f_1$, $N=\partial f_2$,~\eqref{eq:inc_PDHG_gen} reduces to a convex-concave saddle point problem
\begin{equation}\label{eq:PDHG_gen}
\begin{split}
    \inf_x \left\{f_1(x)+f_2(Cx)\right\}=&\inf_x \sup_y \left\{f_1(x)+\langle C x, y \rangle_{\mathcal{H}_2} - f_2(y) \right\}\\
    =&\sup_y \left\{ -f_1^*(-C^* y)-f_2(y)\right\}.
\end{split}
\end{equation}
System~\eqref{eq:inc_PDHG_gen} can be solved by an extension of the celebrated PDHG algorithm:
\begin{equation}\label{eq:PDHG_algo_gen}
    \begin{cases}
    x^{n+1}=(I+\tau M)^{-1} \left( x^n-\tau C^* y^n\right)\\
    \Tilde{x}^{n+1}=2x^{n+1}-x^n\\
    y^{n+1}=(I+\sigma N^{-1})^{-1}\left(y^n+\sigma C \Tilde{x}^{n+1} \right)
    \end{cases}
\end{equation}
where $\tau,\sigma>0$ are such that $\tau \sigma <\frac{1}{\|C\|^2}$.

Our goal is to formulate~\eqref{eq:inc_discrete_1} in the form of~\eqref{eq:inc_PDHG_gen} and solve it via~\eqref{eq:PDHG_algo_gen}.

\begin{theorem}\label{thm:main}
Assume that $F(t,x,\cdot)$, $G(x,\cdot)$ are convex and continuous for $\rho\geq 0$, and denote by
    \begin{equation}\label{eq:MCN}
        \begin{split}
            M(\phi)=&0,~C(\phi)=(A^*\phi,B^* \phi),\quad \phi \in \mathcal{U},\\
            N(\rho,w)=&\mathcal{K}^{-1}(\rho,w),\quad (\rho,w) \in \mathcal{M}\times \mathcal{W}.
        \end{split}
    \end{equation}
    Then the following statements are true.
\begin{enumerate}
    \item $M,N$ are maximally monotone, and~\eqref{eq:inc_discrete_1} is equivalent to the pair of primal-dual inclusions
    \begin{equation}\label{eq:inc_PDHG_our}
        \begin{split}
            \text{find}~\phi~\text{s.t.}~0\in M(\phi)+C^*(N(C\phi))&\quad \text{(P)}\\
            \text{find}~(\rho,w)~\text{s.t.}~(\rho,w)\in N(C\phi),~-C^*(\rho,w) \in M(\phi),~\text{for some}~\phi&\quad \text{(D)}
        \end{split}
    \end{equation}
    \item The problem~\eqref{eq:inc_PDHG_our} admits admits at least one solution.
    \item When $\nu>0$ and $F(t,x,\cdot)$ and $G(x,\cdot)$ are differentiable for $\rho>0$,~\eqref{eq:inc_PDHG_our} is equivalent to~\eqref{eq:mfg_discrete}.
    \item When $\nu>0$ and $F(t,x,\cdot)$ and $G(x,\cdot)$ are differentiable for $\rho>0$, the algorithm
    \begin{equation}\label{eq:PDHG_algo_our}
        \begin{cases}
            \phi^{n+1}=\phi^n-\tau (A\rho^n+Bw^n)\\
            \Tilde{\phi}^{n+1}=2\phi^{n+1}-\phi^n\\
            (\rho^{n+1},w^{n+1})=(I+\sigma \mathcal{K})^{-1}\left(\rho^n+\sigma A^*\Tilde{\phi}^{n+1},w^n+\sigma B^*\Tilde{\phi}^{n+1}\right)
        \end{cases}
    \end{equation}
    yields a solution for~\eqref{eq:mfg_discrete}.
\end{enumerate}
\end{theorem}
\begin{proof}
    \begin{enumerate}
        \item Trivially, we have that $M$ is monotone, and $I+M=I$ is surjective. Hence, by Minty's theorem~\cite[Theorem 3.5.8]{aubin90set} $M$ is maximally monotone. Furthermore, by~\cref{crl:K_max_mon} we have that $\mathcal{K}$ is maximally monotone, and so $N=\mathcal{K}^{-1}$ is also maximally monotone~\cite[Section 3.5.2]{aubin90set}.

        Furthermore, the equivalence of~\eqref{eq:inc_discrete_1} and~\eqref{eq:inc_PDHG_our} follows directly from the definitions of $N$ and $C$ since
        \begin{equation*}
            (\rho,w) \in N(C(\phi)) \Longleftrightarrow (A^*\phi,B^*\phi) \in \mathcal{K}(\rho,w).
        \end{equation*}
        \item By~\cref{lma:fixed_point} we have that there exists $(\rho^*,w^*)$ such that
        \begin{equation*}
            (\rho^*,w^*) \in \underset{A\rho +Bw = 0}{\operatorname{argmin}} J(\rho,w,\rho^*).
        \end{equation*}
        Hence, by part 1 in~\cref{lma:KKT} we have that there exists $\phi^*\in \mathcal{U}$ such that
        \begin{equation*}
            (A^* \phi^*,B^* \phi^*) \in \partial_{(\rho,w)} J(\rho,w,\eta)\Big|_{\eta=\rho^*}=\mathcal{K}(\rho^*,w^*),\quad A\rho^*+Bw^*=0,
        \end{equation*}
        which is precisely~\eqref{eq:inc_discrete_1}. Thus, we conclude by part 1 above.
        \item This follows from part 3 in~\cref{lma:KKT} and part 1 above.
        \item This follows from~\cite{vu13}, and part 3 above.
    \end{enumerate}
\end{proof}

\subsection{Grid-size independent time-steps}
\label{sec:grid-indep}
The convergence of~\eqref{eq:PDHG_algo_our} might be very slow for fine grids. The reason is that $\tau,\sigma>0$ must satisfy is $\tau \sigma <\frac{1}{\|C\|^2}$. But $C(\phi)=(A^*\phi,B^*\phi)$ is a discretization of differential operator $(-\partial_t \phi - \nu \Delta \phi, -\nabla \phi)$ that is unbounded. Thus, $\|C\| \to \infty$ in the continuous limit; that is, when the grid gets finer. Consequently, $\tau,\sigma>0$ should be very small slowing down the convergence.

To remedy this previous issue, we will choose the norm in $\mathcal{U}$ in such a way that $\|C\|=1$, and so $\tau ,\sigma>0$ should just satisfy $\tau \sigma<1$. This procedure will be equivalent to a suitable preconditioning of the $\phi$ update in~\eqref{eq:PDHG_algo_our}.

Consider the following inner product in $\mathcal{U}$:
\begin{equation*}
\begin{split}
    \langle \phi_1, \phi_2 \rangle_\star=&\langle C \phi_1, C \phi_2 \rangle=\langle C^* C \phi_1,\phi_2 \rangle
    = \langle (A A^* +B B^* )\phi_1, \phi_2 \rangle,\quad \forall \phi_1,\phi_2 \in \mathcal{U},
\end{split}
\end{equation*}
where the inner products without $\star$ are the standard ones. Note that $\langle \cdot , \cdot \rangle_\star$ is bilinear, and
\begin{equation*}
    0=\|\phi\|_\star^2=\|C \phi\|^2=\|A^*\phi\|^2+\|B^*\phi\|^2
\end{equation*}
yields $\phi=0$. Hence, $\langle \cdot, \cdot \rangle_{\star}$ is indeed an inner product on $\mathcal{U}$.

Now define $M,C,N$ as before in~\eqref{eq:MCN} but let us calculate the adjoint of $C:\mathcal{U} \to \mathcal{M} \times \mathcal{W}$ with respect to $\langle\cdot, \cdot, \rangle_{\star}$. We have that
\begin{equation*}
    \langle C^*(\rho,w), \phi \rangle=\langle (\rho,w), C \phi \rangle= \langle C_\star^* (\rho,w), \phi \rangle_{\star}= \langle C^* C C_\star^* (\rho,w), \phi \rangle
\end{equation*}
for all $\rho,w,\phi$, and so
\begin{equation*}
    C^*_\star=(C^*C)^{-1}C^*=C^\dagger
\end{equation*}
where $C^\dagger$ is the Moore-Penrose pseudoinverse. Hence,~\eqref{eq:inc_discrete_1} has another equivalent formulation as a primal dual pair of monotone inclusions
\begin{equation}\label{eq:inc_PDHG_our_dagger}
    \begin{split}
        \text{find}~\phi~\text{s.t.}~M(\phi)\in C^\dagger(N(C\phi))&\quad \text{(P)}\\
        \text{find}~(\rho,w)~\text{s.t.}~(\rho,w)\in N(C\phi),~-C^\dagger(\rho,w)\in M(\phi),~\text{for some}~\phi&\quad \text{(D)} 
    \end{split}
\end{equation}
Accordingly, we have the following PDHG algorithm
\begin{equation}\label{eq:PDHG_algo_our_dagger}
    \begin{cases}
    \phi^{n+1}=\phi^n-\tau (AA^*+BB^*)^{-1}(A\rho^n+Bw^n)\\
    \Tilde{\phi}^{n+1}=2\phi^{n+1}-\phi^n\\
    (\rho^{n+1},w^{n+1})=(I+\sigma \mathcal{K})^{-1}\left(\rho^n+\sigma A^*\Tilde{\phi}^{n+1},w^n+\sigma B^*\Tilde{\phi}^{n+1}\right)
    \end{cases}
\end{equation}

Note that the only difference between~\eqref{eq:PDHG_algo_our_dagger} and~\eqref{eq:PDHG_algo_our} is the $\phi$ update step, which is preconditioned with $C^*C=AA^*+BB^*$. The advantage is that
\begin{equation*}
    \|C\|_\star=\sup_{\|\phi\|_\star \leq 1} \|C\phi \|=\sup_{\|C\phi\| \leq 1} \|C\phi \| =1,
\end{equation*}
and so $\tau,\sigma>0$ in~\eqref{eq:inc_PDHG_our_dagger} must satisfy $\tau\sigma <1$. Therefore, the magnitude of $\tau,\sigma$ and the convergence rate of the algorithm are independent of the grid.

\subsection{Further remarks on the algorithm}

Note that $C^*C=AA^*+BB^*$ is the discretization of the continuous operator
\begin{equation*}
    (-\partial_t-\nu \Delta)(\partial_t-\nu \Delta)+\nabla \cdot (-\nabla)=-\partial_t^2-\Delta +\nu^2 \Delta^2 \geq 0
\end{equation*}
Hence, the $\phi$ update in~\eqref{eq:inc_PDHG_our_dagger} invloves a solution of fourth-order space-time PDE that can be done via FFT. 

Furthermore, the $(\rho,w)$ update in~\eqref{eq:PDHG_algo_our_dagger} decouples to $3$-dimensional, or $(d+1)$-dimensional for $d\neq 2$, monotone inclusion problems at grid points that can be solved efficiently by standard algorithms. Often, $w$ can be eliminated leading to one-dimensional problems with respect to $\rho$ at the grid points. This happens, for example, when solving optimal transport or potential MFG problems~\cite{BenamouBrenier2000,bencar'15,bencarsan'17,silva18,silva19} or the congestion model~\cite{achdou18} considered here.

When the convex duals or $F(t,x,\cdot)$ and $G(x,\cdot)$ are readily available, we can simply the $(\rho,w)$ updates by a further splitting. To this end, assume that $F(t,x,\cdot)$ and $G(x,\cdot)$ are convex and continuous in $\rho\geq 0$, and
\begin{equation*}
    F(t,x,\rho)=\infty,\quad G(x,\rho)=\infty,\quad \forall \rho<0,
\end{equation*}
and consider
\begin{equation}\label{eq:F*G*}
    F^*(t,x,a)=\sup_{\rho \geq 0}\left\{ a\cdot \rho-F(t,x,\rho) \right\},\quad G^*(x,b)=\sup_{\rho \geq 0}\left\{ b\cdot \rho-G(x,\rho) \right\},
\end{equation}
which are proper convex lower semicontinuous functions with respect to $a,b$, respectively. Next, consider 
\begin{equation}\label{eq:calFG}
\begin{split}
    \mathcal{F}^*(a)=\sum\limits_{\substack{0\leq i,j \leq N_h-1\\ 0\leq k\leq N_T-1}} F^*(t_k,x_{ij},a_{ij}^{k}),\quad \mathcal{G}^*(b)=\sum\limits_{0\leq i,j \leq N_h-1} \frac{G^*(x_{ij},b_{ij} \Delta t )}{\Delta t}.
\end{split}
\end{equation}
Note that $\mathcal{F}^*,\mathcal{G}^*$ are also proper convex lower semicontinuous functions. Furthermore, denote by
\begin{equation}\label{eq:hatJ}
\begin{split}
    \widetilde{J}(\rho,w,\eta)=&\sum\limits_{\substack{0\leq i,j \leq N_h-1\\ 0\leq k\leq N_T-1}} E_h \left(t_k, x_{ij},w^k_{ij},\rho^{k+1}_{ij},\eta^{k+1}_{ij}\right)+\mathbf{1}_{\rho^0=\rho_0},
\end{split}
\end{equation}
and consider
\begin{equation*}
    \begin{split}
    \widetilde{\mathcal{K}}(\rho,w)=&\partial_{(\rho,w)} \widetilde{J}(\rho,w,\eta)\Big|_{\eta=\rho},\quad (\rho,w) \in \mathcal{M} \times \mathcal{W},\\
    \widetilde{N}(\rho,w)=&\widetilde{\mathcal{K}}^{-1}(\rho,w), \quad (\rho,w) \in \mathcal{M} \times \mathcal{W},\\
    \widetilde{M}(a,b,\phi)=&\left(\partial_a \mathcal{F}^*(a),\partial_b \mathcal{G}^*(b),~0\right),\quad (a,b,\phi)\in \mathcal{U} \times \mathbb{R}^{N_h\times N_h}\times  \mathcal{U}.
    \end{split}
\end{equation*}
Finally, we define $\widetilde{C}:\mathcal{U} \times \mathbb{R}^{N_h\times N_h} \times \mathcal{U} \to \mathcal{M} \times \mathcal{W}$ as
\begin{equation*}
    \widetilde{C}(a,b,\phi)=\left(A^*\phi-(0,a^{-(N_T-1)},a^{N_T-1}+b),B^*\phi\right),
\end{equation*}
where
\begin{equation*}
    a^k=(a^k_{ij})_{i,j=0}^{N_h-1},\quad a^{-k}=(a^l_{ij})_{i,j=0,l\neq k}^{N_h-1}.
\end{equation*}
Hence, we have that $\widetilde{C}^*:\mathcal{M} \times \mathcal{W} \to \mathcal{U} \times \mathbb{R}^{N_h \times N_h}\times \mathcal{U}$ is given by
\begin{equation*}
    \widetilde{C}^*(\rho,w)=(-\rho^{-0},-\rho^{N_T},A\rho+Bw).
\end{equation*}

Then~\eqref{eq:inc_discrete_1} can be written as
\begin{equation}\label{eq:inc_PDHG_our_simple}
    \begin{split}
        \text{find}~(a,b,\phi)~\text{s.t.}~0\in \widetilde{M}(a,b,\phi)+\widetilde{C}^*(\widetilde{N}(\widetilde{C}(a,b,\phi)))&\quad \text{(P)}\\
        \text{find}~(\rho,w)~\text{s.t.}~(\rho,w)\in \widetilde{N}(\widetilde{C}(a,b,\phi)),-\widetilde{C}^*(\rho,w)\in \widetilde{M}(a,b,\phi),~\text{for some}~(a,b,\phi)&\quad \text{(D)} 
    \end{split}
\end{equation}
Consequently, the PDHG algorithm is
\begin{equation}\label{eq:PDHG_algo_our_simple}
    \begin{cases}
    a^{n+1}=(I+\tau \partial_a \mathcal{F}^*)^{-1}\left(a^n+\tau (\rho^n)^{-0}\right)\\
    b^{n+1}=(I+\tau \partial_b \mathcal{G}^*)^{-1}\left(b^n+\tau (\rho^n)^{N_T}\right)\\
    \phi^{n+1}=\phi^n-\tau (A\rho^n+Bw^n)\\
    (\Tilde{a}^{n+1},\Tilde{b}^{n+1},\Tilde{\phi}^{n+1})=2(a^{n+1},b^{n+1},\phi^{n+1})-(a^{n},b^{n},\phi^{n})\\
    (\rho^{n+1},w^{n+1})=(I+\sigma \widetilde{\mathcal{K}})^{-1}\left((\rho^n,w^n)+\sigma \widetilde{C}(\Tilde{a}^{n+1},\Tilde{b}^{n+1},\Tilde{\phi}^{n+1})\right)
    \end{cases}
\end{equation}
The advantage of~\eqref{eq:PDHG_algo_our_simple} over~\eqref{eq:PDHG_algo_our} is that the $(\rho,w)$ update simplifies due to a simpler function $\widetilde{\mathcal{K}}$ at the expense of often simple proximal updates of dual variables $a,b$. As before, except the $\phi$ update all other updates are decoupled at the grid points.

Additionally, we can obtain a grid independent convergence rate by introducing an inner product
\begin{equation*}
    \langle (a_1,b_1,\phi_1), (a_2,b_2,\phi_2) \rangle_\star=\langle a_1, a_2 \rangle+\langle b_1,b_2\rangle +\langle (AA^*+BB^*)\phi_1,\phi_2 \rangle,
\end{equation*}
which leads
\begin{equation*}
    \widetilde{C}^*_\star=\left( -\rho^{-0},-\rho^{N_T},(AA^*+BB^*)^{-1}(A\rho+Bw)\right)
\end{equation*}
and a preconditioned version of~\eqref{eq:PDHG_algo_our_simple}:
\begin{equation}\label{eq:PDHG_algo_our_dagger_simple}
    \begin{cases}
    a^{n+1}=(I+\tau \partial_a \mathcal{F}^*)^{-1}\left(a^n+\tau (\rho^n)^{-0}\right)\\
    b^{n+1}=(I+\tau \partial_b \mathcal{G}^*)^{-1}\left(b^n+\tau (\rho^n)^{N_T}\right)\\
    \phi^{n+1}=\phi^n-\tau (AA^*+BB^*)^{-1} (A\rho^n+Bw^n)\\
    (\Tilde{a}^{n+1},\Tilde{b}^{n+1},\Tilde{\phi}^{n+1})=2(a^{n+1},b^{n+1},\phi^{n+1})-(a^{n},b^{n},\phi^{n})\\
    (\rho^{n+1},w^{n+1})=(I+\sigma \widetilde{\mathcal{K}})^{-1}\left((\rho^n,w^n)+\sigma \widetilde{C}(\Tilde{a}^{n+1},\Tilde{b}^{n+1},\Tilde{\phi}^{n+1})\right)
    \end{cases}
\end{equation}

\section{Numerical Experiments}
In this section, we present two numerical experiments to demonstrate the effectiveness of the proposed numerical scheme and its robustness with respect to singular limits $\nu=0,\epsilon=0$. We conduct the following experiments on a $2$-dimensional torus $\mathbb{T}^2 = [0,1] \times [0,1]$ and a time interval $[0,T]=[0,1]$.

\subsection{Convergence of the algorithm}\label{sec:num_eg1}
In this example, we numerically verify the grid-independent time steps as discussed in~\cref{sec:grid-indep}.
We apply the proposed monotone PDHG approach to the following mean-field game system with congestion with
$\alpha = 1, \beta = 2$ in~\cref{eq:H}, $\epsilon = 0.1, \nu = 0.1$. The initial distribution is given as follows 
\begin{align*}
    &\rho_0(x) = \hat{\rho}_{c(0.25, 0.25)}(x)+ \hat{\rho}_{c(0.75, 0.25)}(x) + \hat{\rho}_{c(0.25, 0.75)}(x) + \hat{\rho}_{c(0.75, 0.75)}(x),\\
    &\hat{\rho}_{c(a_1,a_2)}(x)= \hat{c} \exp\left( -100\left(( x_1-a_1)^2 + ( x_2-a_2)^2\right)\right),
\end{align*}
where $\hat{c}>0$ is a constant scalar such that $\int \rho_0(x) dx = 1$.
For the mean-field game system, the terminal condition $g$ is given by
\begin{align*}
    g(x) = 0.1 \left( \sin\left(\pi \left(x_1-0.5\right)\right)\right)^2 +  0.1 \left( \sin\left(2\pi \left(x_2-0.25\right)\right)\right)^2.
\end{align*}
In~\cref{fig:example1}, we show the numerical results with discretized grid $N_h =40, N_t = 32$. We observe that the density moves towards points $(0.5,0.25)$ and $(0.5,0.75)$, which are the minima of the terminal cost function $g$. We also see the diffusion during the time evolution that is induced by the viscosity parameter $\nu = 0.1$. 

\begin{figure}
    \centering
    \includegraphics[width = 1.0\textwidth,trim=115 75 120 100 , clip]{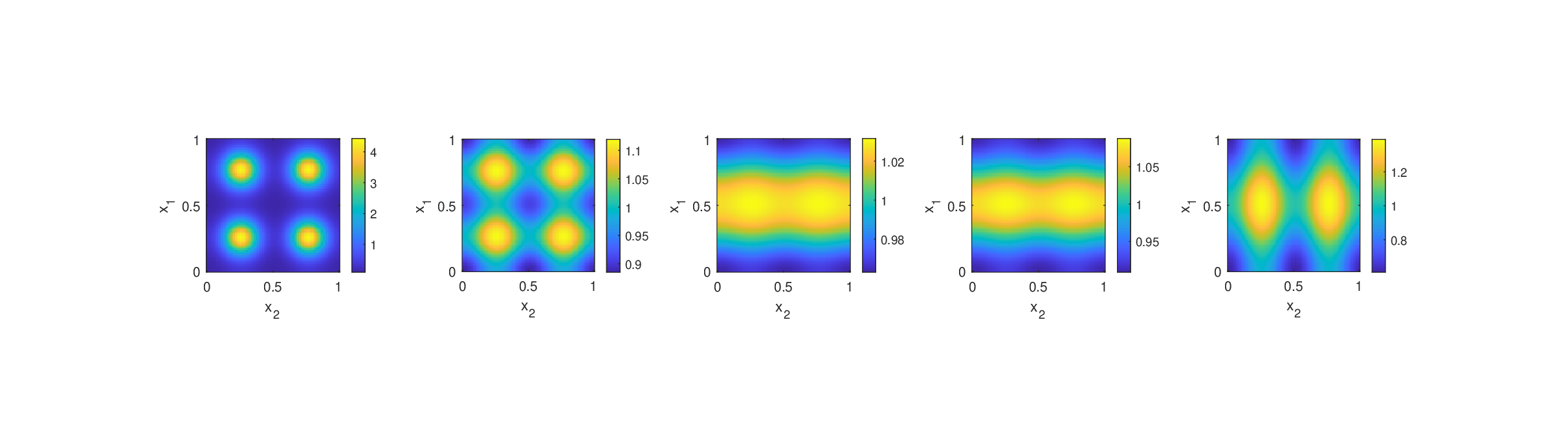}
    \caption{This figure shows the density evolution of $\rho(x,t)$ in the MFG model with congestion parameters $\alpha = 1, \beta = 2, \epsilon = 0.1$ and viscosity $\nu=0.1$, as detailed in~\cref{sec:num_eg1}. Moving from left to right, at time points $t = 0, 0.25, 0.5, 0.75, 1$, the density starts in the form of $4$ Gaussians and eventually concentrates at points $(0.5, 0.25)$ and $(0.5, 0.75)$.}
    \label{fig:example1}
\end{figure}

To show that the convergence rate is independent of the grid-size, we set error tolerance $\varepsilon_1 = 10^{-3}, \varepsilon_2 = 10^{-4}, \varepsilon_3 = 10^{-5}$, and compute the number of iterations the algorithm needs to guarantee that the discretized continuity equation~\cref{eq:KFP_discrete} satisfies
\[
\text{Err}^n := \| \text{discretized continuity equation at $n$-th iteration } \|_{L^2} < \varepsilon.
\]
Here, the norm $\| \cdot\|_{L^2}$ has been properly scaled with the spatial and temporal discretization $\Delta h, \Delta t$. We perform computations for various grid-sizes and summarize the results in~\cref{tab:example1_convergence}.

We fixed the optimization step-sizes $\sigma = \tau  = 0.5$. Although preconditioning increases the computational cost per iteration, it is important to emphasize that it ultimately enhances the computational efficiency. To this end, we consider a scenario without preconditioning with $N_h = 20$ and $N_t = 16$, and the algorithm needs $59723$ iterations (taking $508.73$ seconds) to achieve a residual error smaller than $\varepsilon_1 = 10^{-3}$, and $110093$ iterations (taking $932.74$ seconds) to reach a residual error smaller than $\varepsilon_2 = 10^{-4}$. This example highlights how suitable choices of norms and inner products in the PDHG improves the overall efficiency of the proposed algorithm.

\begin{figure}
        \centering
  \begin{tabular}{|c|c|c|c|c|}
     \hline
        $N_h$ & $N_t$ & $\varepsilon_1 = 10^{-3}$ & $\varepsilon_2 = 10^{-4}$ & 
        $\varepsilon_3 = 10^{-5}$ \\
         \hline
        20 & 16 & 107 (1.34 s) & 137 (1.69 s) &  196 (2.37 s)\\
         \hline
        32 & 20 & 117 (4.67 s) &  147 (5.83 s) & 197 (7.77 s)\\
         \hline
        40 & 32& 117 (11.19 s) & 157 (14.62 s) & 197 (18.57 s) \\
         \hline
    \end{tabular}
        \caption{We record the number of iterations our algorithm requires to reach specific error thresholds, $\varepsilon$, across various mesh densities, along with the computational time needed (in seconds). Since the algorithm is not optimized for parallel execution, the computational time increases with the refinement of the mesh.}
        \label{tab:example1_convergence}
\end{figure}

\subsection{Viscosity Effect}\label{subsec:viscosity=0}

In this section, we conduct experiments to demonstrate the impact of the viscosity parameter $\nu$ in the MFG model with congestion. The parameters $\alpha,\beta,\epsilon$ are the same as in the previous section, and we compute the solutions for $\nu=0,0.02,0.1$. The computations demonstrate that $\nu$ significantly affects the solution, as depicted in~\cref{fig:example2}. A stronger diffusion results in a more widespread solution. More importantly, the algorithm is robust with respect to small values of $\nu$. This phenomenon is explained by the variational nature of the algorithm that can handle singular problems.
\begin{figure}
    \centering
    \includegraphics[width = 1.0\textwidth,trim=60 35 70 55, clip]{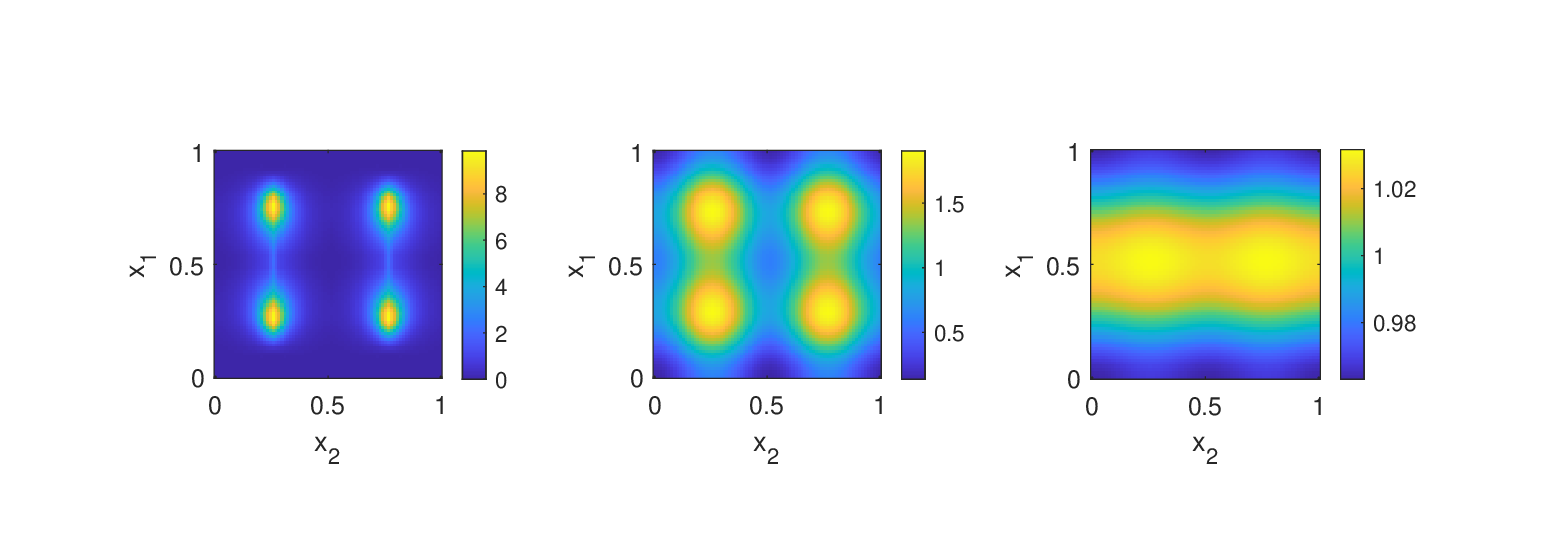}
        \includegraphics[width = 1.0\textwidth,trim= 60 35 70 55 , clip]{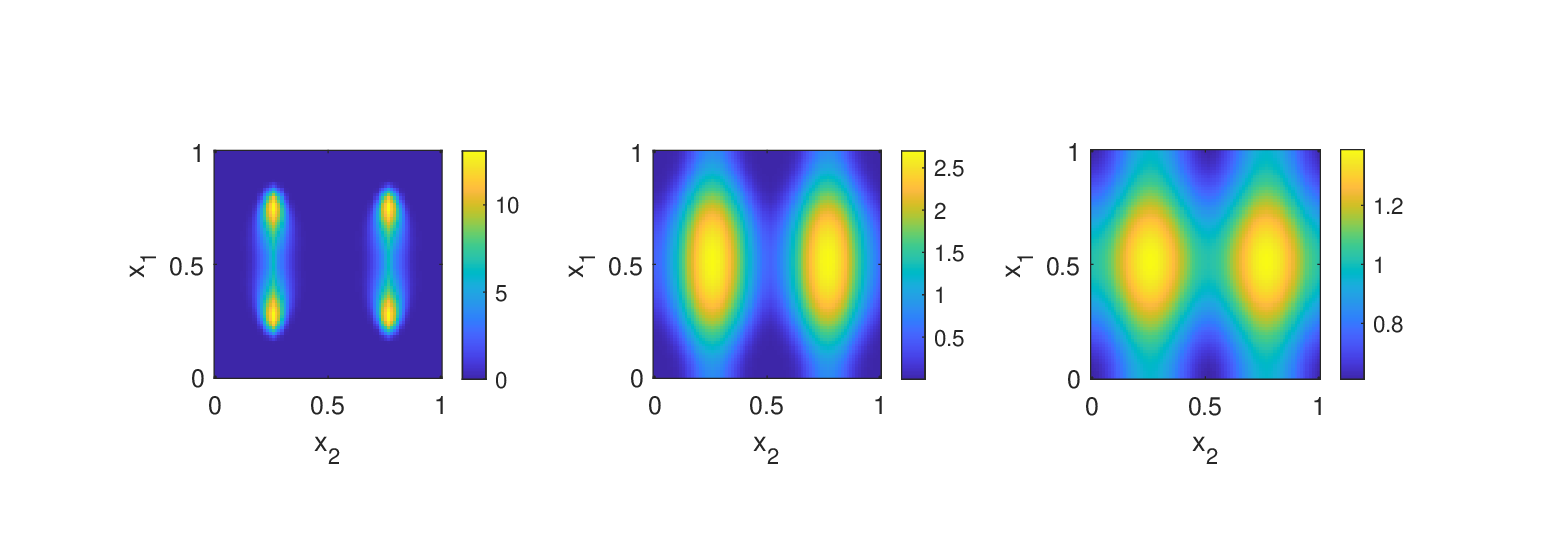}
    \caption{This figure shows the numerical results of density evolution of congestion model with $\nu = 0, 0.02, 0.1 $ (from left to right): $\rho(0.5,x)$ (top) and $\rho(1,x)$ (bottom).}
    \label{fig:example2}
\end{figure}

\subsection{Congestion Effect}\label{subsec:cong=0}

In this section, we explore scenarios where $\nu = 0$ and the congestion parameter is small $\epsilon$. The goal of these experiments is to study the robustness of the algorithm with respect to further singularity introduced by small values of $\epsilon$. \cref{fig:example3} illustrates that with an increasing congestion coefficient $\epsilon$, the density tends to exhibit reduced movement. Consequently, when $\epsilon = 5$, the density undergoes only slight deformation. Additionally, as $\epsilon \rightarrow 0$, the solution consistently demonstrates behavior akin to the leftmost case with $\epsilon = 0$. As before, the algorithm demonstrates robust performance in all cases.

\begin{figure}
    \centering
    \includegraphics[width = 1.0\textwidth,trim=160 35 150 55, clip]{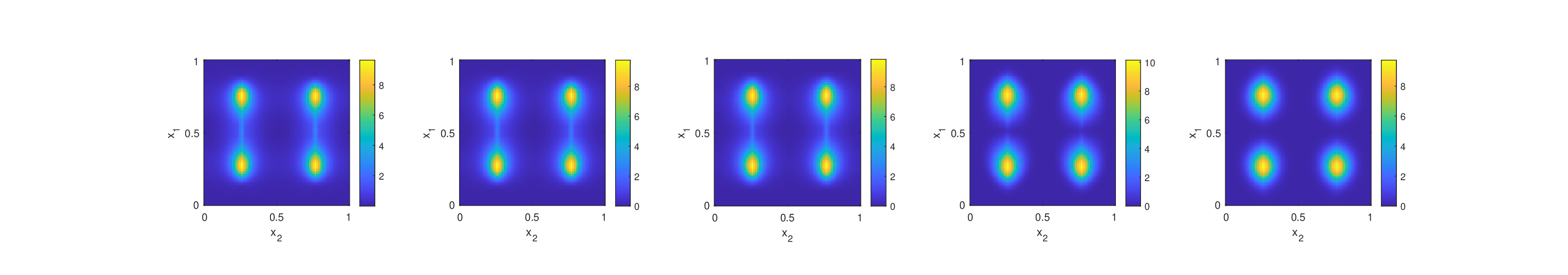}
        \includegraphics[width = 1.0\textwidth,trim= 160 35 150 45, clip]{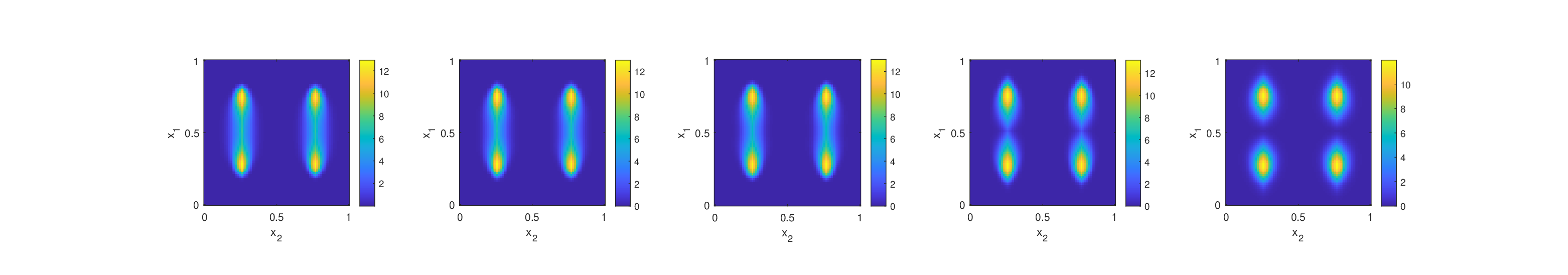}
    \caption{This figure shows the numerical results of density evolution of congestion model with $\epsilon = 0, 0.02, 0.1, 1, 5 $ (from left to right): $\rho(0.5,x)$ (top) and $\rho(1,x)$ (bottom).}
    \label{fig:example3}
\end{figure}




\bibliographystyle{siamplain}

\end{document}